\documentclass{amsart}
\usepackage{amssymb,color,mathrsfs
}

    \def\cD{\mathcal D} \def\cE{\mathcal E}               \def\cS{\mathcal S}    

\def\sF{\mathscr F}      
\def\sA{\mathscr A}

\def\fw{\mathfrak{w}} 
\def\fm{\mathfrak{m}}
\def\fp{\mathfrak{p}}

\def\ff{\mathfrak{f}}
\def\fa{\mathfrak{a}}

\def\fp{\mathfrak{p}}

 \def\fn{\mathfrak{n}}

\def\Z{{\mathbb Z}} \def\R{{\mathbb R}} \def\F{{\mathbb F}} \def\N{{\mathrm N}}  \def\Q{{\mathbb Q}}

\newcommand{\SF}{\text{ sqf}}

\def \ring {\Z[\omega]}
\newcommand{\kommentar}[1]{}

\DeclareMathOperator{\tr}{Tr}

\DeclareMathOperator{\res}{Res}

\DeclareMathOperator{\re}{Re}

\DeclareMathOperator{\cond}{cond}

\def \e{e}
\def\ef{\cE_{\sF_3}}
\def\eft{\cE_{\sF_3'}} 

\renewcommand{\mod}[1]{\,{\rm mod}\,#1}

\def\su#1{\sum_{\substack{#1}}}
\def\pr#1{\prod_{\substack{#1}}}

\def\bs#1{\begin{equation*} \begin{split} #1 \end{split} \end{equation*}}
\def\bsc#1{\begin{equation} \begin{split} #1 \end{split} \end{equation}}
\def\eqs#1{\begin{equation*} #1 \end{equation*}}
\def\eqn#1{\begin{equation} #1 \end{equation}}
\def\mult#1{\begin{multline*}#1\end{multline*}}
\def\multn#1{\begin{multline}#1\end{multline}}
\def\ar#1{\left\{ \begin{array}{l@{\quad\text{if }}l} #1 \end{array}\right.
}

\def\({\left(} \def\){\right)} \def\[{\left[} \def\]{\right]} 
 
\def\le{\leqslant} \def\ge{\geqslant}
\def\eps{{\varepsilon}}

\definecolor{pink}{rgb}{1,.2,.6}
\definecolor{orange}{rgb}{0.7,0.3,0}
\definecolor{blue}{rgb}{.2,.6,.75}
\definecolor{green}{rgb}{.4,.7,.4}
\definecolor{purple}{RGB}{127,0,255}

\newcommand{\ccom}[1]{{\color{pink}{CD: #1}} }

\newtheorem{lem}{Lemma}[section]
\newtheorem{prop}[lem]{Proposition}
\newtheorem{thm}[lem]{Theorem}

\newtheorem{cor}[lem]{Corollary}

\newtheorem{conjecture}[lem]{Conjecture}

\theoremstyle{definition}

\newtheorem{rem}[lem]{Remark}

\newcommand{\phat}{\widehat{\phi}}

\begin{document}

\author{Chantal David}
\address{Department of Mathematics and Statistics, Concordia University, 1455 de Maisonneuve West, Montr\'eal, Qu\'ebec, Canada H3G 1M8}
\email{chantal.david@concordia.ca}

\author{Ahmet M. G\"ulo\u{g}lu  }
\address{Department of Mathematics, Bilkent University, Ankara, Turkey}
\email{guloglua@fen.bilkent.edu.tr}

\thanks{Both authors are supported by T\"UB\.ITAK Research Grant no. 119F413}

\title[One-level density and non-vanishing  for cubic $L$-functions]{One-level density and non-vanishing  for cubic $L$-functions over the Eisenstein field}

\maketitle

\begin{abstract}
We study the one-level density for families of L-functions associated with  cubic Dirichlet characters defined over the Eisenstein field.  We show that the family of $L$-functions associated with the cubic residue symbols $\chi_n$ with $n$ square-free and congruent to 1 modulo 9 satisfies the Katz-Sarnak conjecture for all test functions whose Fourier transforms are supported in $(-13/11, 13/11)$, under GRH. This is the first result extending the support outside the {\it trivial range} $(-1, 1)$ for a family of cubic L-functions.
This implies that a positive density of the L-functions associated with these characters do 
 not vanish at the central point $s=1/2$. 
 A key ingredient in our proof is a bound on an average of generalized cubic Gauss sums at prime arguments, whose proof is based on the work of Heath-Brown and Patterson \cite{HBP1979, HB2000}.
\end{abstract}

\section{Introduction}
Let $\sF$ be a family of primitive Dirichlet characters $\chi$ defined over $\Q$, or more generally over a number field $K$. From the work of Dirichlet and Hecke, we know that the $L$-functions $L(s,\chi)$ satisfy a functional equation relating the values of  $L(s,\chi)$ to those of $L(1-s, \overline{\chi})$, and the distribution of the non-trivial zeros of $L(s,\chi)$ in the central critical strip is of particular interest.

The one-level density for the family $\sF$ measures the density of the low-lying zeros (i.e. the zeros near $s=1/2$) of the L-functions associated with the characters in $\sF$. Following the work of Montgomery \cite{Mont}, and then Katz and Sarnak \cite{KS-bulletin, KS-book}, we believe that the statistics of the low-lying zeros of these $L$-functions match those of the eigenvalues of random matrices in a certain symmetry group associated with the family $\sF$, usually symplectic, orthogonal, or unitary. 

Let $\phi$ be an even Schwartz test function. For a fixed character in $\sF$, the sum
\eqs{\su{\rho=1/2+i\gamma\\ L(\rho, \chi) = 0} \phi \Bigl( \frac{\gamma \log X}{2 \pi} \Bigr)}
counts, with multiplicity, the zeros of $L(s, \chi)$ that are within $O(1/\log X)$ of the central critical point $s=1/2$. To study the statistics of these zeros, one has to consider the average over the family. In this paper, we consider smoothed averages. Let $w:\R \to (0,\infty)$ be an even Schwartz  function, and 
\begin{align*} \cD(X;\phi, \sF) &= \frac 1 {\sA_\sF (X)} \sum_{\chi \in \sF} w\left( \frac{\N(\cond(\chi))}{X} \right)  \su{\gamma \\ L(1/2+i\gamma, \chi) = 0}\phi \Bigl( \frac{\gamma \log X}{2 \pi} \Bigr)\\
\sA_\sF (X) &= \sum_{\chi \in \sF} w\left( \frac{\N(\cond(\chi))}{X} \right), 
\end{align*}
where $\N(\cond(\chi))$ is the norm of the conductor of the primitive character $\chi$.
The one-level density is then defined as
\eqs{\lim_{X \rightarrow \infty} \cD(X;\phi, \sF).}

\begin{conjecture} [Katz-Sarnak \cite{KS-book,KS-bulletin}] \label{conjecture-KS} With the notation above, we have
\eqs{\lim_{X \rightarrow \infty} \cD(X; \phi, \sF) = \int_{-\infty}^\infty \phi(x) \;W_G(x) \; dx,}
where $W_G(x)$ measures one-level density of eigenvalues near 1 of the classical compact group $G=G({\sF})$ corresponding to the symmetry type of the family
$\sF$. 
\end{conjecture} 

We refer the reader to \cite[page 409]{KS-book} for the precise formulae of the densities $W_G(x)$ for the different symmetry groups $G$.

The conjecture of Katz and Sarnak is still open, but evidence for the conjecture can be obtained by proving that the conjecture holds for test functions $\phi$ whose Fourier transforms $\phat$ 
have compact support. 

Assuming GRH, \"Ozl\"uk and Snyder \cite{OS-1999} showed that the one-level density for the family of quadratic characters satisfies the Katz-Sarnak conjecture with symplectic symmetry for test functions $\phi$ with $\phat$ supported in $(-2, 2)$. The same result was obtained over functions fields by Rudnick \cite{Rudnick} and Bui and Florea \cite{BuiFlo2018} (who also identify some lower order terms).

Besides the family of quadratic Dirichlet L-functions, Conjecture \ref{conjecture-KS} has been confirmed (with limited support) for many other families of L-functions, such as different types of Dirichlet L-functions \cite{DPR2020, Gao,HR2003,M2008,OS-1999, R2001}, L-functions with characters of the ideal class group of the imaginary quadratic field $\Q(\sqrt{-D})$ \cite{FI2003}, automorphic L-functions \cite{HM2007,HM2009,ILS,RR2011,Royer}, elliptic curve L-functions \cite{BZ2008,B1992, DHP2015, HB2004,M2004,Y2005}, Hecke L-functions for characters of infinite order \cite{Waxman2021}, symmetric powers of GL(2)
L-functions \cite{DM2006,G2005}, and a family of GL(4) and GL(6) L-functions \cite{DM2006}.

We study in this paper the one-level density for families of primitive cubic Dirichlet characters defined over the Eisenstein field $\Q(\omega)$, where $\omega=e^{2\pi i/3}$. Many new conceptual and technical difficulties appear when considering cubic (and not quadratic) characters, and the results in the literature are fewer and weaker. Meisner \cite{Meisner}  and Cho and Park \cite{ChoPark} showed that, under GRH, the one-level density for families of cubic characters over $\Q$ satisfies the Katz-Sarnak conjecture with unitary symmetry for test functions $\phi$ with $\phat$ supported in $(-1, 1)$. Unconditional (and weaker) results were obtained by Gao and Zhao \cite{GZ2011} and Meisner \cite{Meisner} .

The $(-1,1)$-range for the support of $\phat$ is a natural boundary for families attached to Dirichlet characters, and our work provides the first example of a family of cubic L-functions in which the support is  extended past this trivial range, assuming GRH. In the recent work \cite{DPR2020}, Drappeau, Pratt, and Radziwill computed the one-level density over the family of all primitive Dirichlet characters, and proved the first unconditional result extending the trivial range for this family.

An important tool to extend the support past the trivial range for quadratic characters is the use of Poisson summation.  This is more intricate for cubic characters, as Poisson summation leads to averages of Gauss sums, and while quadratic Gauss sums are given by a simple formula, their cubic analogues exhibit chaotic behaviour; in order to understand them, we must invoke the deep work of Kubota and Patterson. 
Furthermore, other features of the cubic families seem to conspire to make this strategy fail: for cubic characters over $\Q$, the Gauss sums are not defined over the ground field, and for cubic characters over $\Q(\omega)$, there are ``too many characters". For this reason, it had then become customary in the literature to consider a thin subfamily of the cubic characters 
over $\Q(\omega)$,  as in \cite{GZ2011, Luo, FHL, Diaconu}. 
Only recently, moments for the whole family 
of cubic characters when the base field contains the cubic root of unity 
 were considered in the work of David, Florea and Lalin, who computed the first moment for the whole family over function fields \cite{DFL-1}.

\subsection{Statement of the results}
The main result of this paper is the following theorem, where we extend the support of the Fourier transform of the test function for the thin family $\sF_3'$.

\begin{thm} \label{thm-main} Let $\sF_3'$ be the family of primitive cubic Dirichlet characters defined by \eqref{thinfamily}. Let $\phi$ be an even Schwartz function with $\phat$ supported in $(-\frac{13}{11}, \frac{13}{11})$. If GRH holds for $L(s,\chi)$ for each $\chi \in \sF_3'$, then  
\eqs{\lim_{X \rightarrow \infty} \cD(X;\phi, \sF_3') = \int_{-\infty}^\infty \phi(x) W_{U}(x) dx = \phat(0),}
where $W_{U}(x)$ is the kernel measuring the one-level density for the eigenvalues of unitary matrices. 
\end{thm}

A folklore conjecture of Chowla predicts that
$L(\frac12, \chi) \neq 0$ for {\it all} L-functions $L(s, \chi)$ attached to Dirichlet characters. Over function fields, this is false, and in a recent paper, Li \cite{Li} showed that there are infinitively many quadratic Dirichlet L-functions such that $L(\frac12, \chi) = 0$ in this case. It is believed that the number of such characters should be of density zero among all quadratic characters, which is implied by  (the function field version of) the conjecture of Katz and Sarnak.
 It is well-known that proving Conjecture \ref{conjecture-KS}  for test functions where the support of $\phat$ is large enough, yields a \emph{positive proportion} of non-vanishing for the corresponding set of L-functions, bringing evidence to Chowla's conjecture. For the family of cubic characters, one needs to extend the support beyond $(-1,1)$ to get a positive proportion. Hence, by Theorem \ref{thm-main},  we can prove the following result for the thin subfamily of cubic characters.

\begin{cor} \label{non-vanishing} \label{coro-non-vanishing} 
Let $\sF_3'$ be the family of primitive cubic Dirichlet characters defined in \eqref{thinfamily}. 
If GRH holds for the corresponding L-functions, then $L(\tfrac12, \chi) \neq 0$ for at least $2/13$ of the characters in $\sF_3'$.
\end{cor}

Our result is the first result showing a positive proportion of non-vanishing for any cubic family over number fields. Unconditionally, it is known that there are infinitely many cubic Dirichlet characters $\chi$ such that $L(\frac12, \chi) \neq 0$, over $\Q$ \cite{BaYo} and over $\Q(\omega)$ \cite{Luo}.
Over function fields, 
a positive proportion of non-vanishing was obtained by David, Florea and Lalin \cite{DFL-2}
 for the family of cubic characters when $\F_q$ does not contain a third root of unity (which is the equivalent of cubic characters over $\Q$), improving previous work \cite{DFL-1, ELS2020} exhibiting infinitively many cubic Dirichlet characters $\chi$ such that $L(\frac12, \chi) \neq 0$ over function fields.
 The proof of \cite{DFL-2} uses a completely different technique from this paper, based on the mollified moments. It is interesting to compare the techniques and results, and  we believe that the results of \cite{DFL-2} could be obtained for cubic characters over $\Q$, where the one-level density approach seems to fail. 
 Of course, we would then need to assume GRH (which is proven over function fields). We also speculate that the mollified moments approach is more likely to succeed in obtaining a positive proportion of non-vanishing for the full family of primitive cubic characters over $\Q(\omega)$ than the one-level density approach, as breaking the $(-1,1)$-barrier using the one-level density is harder due to the size of the family. As there is no result in the literature for the one-level density for the full family over $\Q(\omega)$, we include the following result, which supports the Katz-Sarnak conjecture for test functions whose Fourier transform has support in the trivial range $(-1,1)$.

\begin{thm}\label{thm-trivial} Let $\sF_3$ be the family of primitive cubic Dirichlet characters defined in Section \ref{full-family}. Let $\phi$ be an even Schwartz function with $\phat$ supported in $(-1, 1)$. Assume GRH for $L(s,\chi)$ for each $\chi \in \sF_3$. Then,
\eqs{\lim_{X \to \infty} \cD(X;\phi, \sF_3) = \int_{-\infty}^\infty \phi(x) W_{U}(x) dx,}
where $W_{U}(x)$ is the kernel measuring the one-level density for eigenvalues of unitary matrices. 
\end{thm}

Finally, we state some unconditional results.
\begin{thm}\label{thm-non-GRH} 
Let $\sF_3$ be the family of primitive cubic Dirichlet characters over $\Q(\omega)$ defined in Section \ref{full-family}. Let $\phi$ be an even Schwartz function with $\phat$ supported in $(-\frac12, \frac12)$. Then,
\eqs{\lim_{X \to \infty} \cD(X;\phi, \sF) = \int_{-\infty}^\infty \phi(x) W_{U}(x) dx = \phat(0),}
where $W_{U}(x)$ is the kernel measuring the one-level density for eigenvalues of unitary matrices. The same result holds for the subfamily $\sF_3'$ defined in Section \ref{section-thm-main} with $\phat$ supported in $(-\frac23, \frac23)$.
\end{thm} 

We remark that since we consider smooth sums, the support of $\phat$ for the family $\sF_3'$ in Theorem \ref{thm-non-GRH}
is slightly better than the one obtained by Gao and Zhao in \cite{GZ2011}, which requires supp$(\phat) \subseteq (-\frac35, \frac35)$ for the same family.

\subsection{Structure of the paper}
 In Section \ref{section-cubic}, we collect the relevant facts about cubic characters, cubic families and cubic Gauss sums. In Section \ref{full-family}, we prove Theorem \ref{thm-trivial}.
In Section \ref{section-thm-main}, we use Poisson summation to reduce the computation of the one-level density to averages of generalized cubic Gauss sums at prime arguments, and we prove
Theorem \ref{thm-main}, assuming the bounds for those averages given by Theorem \ref{average-GS}. We prove Theorem \ref{average-GS} in Section \ref{section-generalized-HB}
and Section \ref{section6} by generalising the work of Heath-Brown \cite{HB2000} and Heath-Brown and Patterson \cite{HBP1979} on the distribution of cubic Gauss sums at prime arguments. The proof of Corollary \ref{non-vanishing} is given in Section \ref{section-non-vanishing}.

\section{$L$-functions of cubic characters and cubic Gauss sums}
\label{section-cubic}

\subsection{Cubic Dirichlet $L$-functions}
Let $K= \Q(\omega)$, $\omega=e^{2 \pi i/3}$. The ring of integers $\ring$ of $K$ has class number one and six units $\left\{ \pm 1, \pm \omega, \pm \omega^2 \right\}$. Each non-trivial principal ideal $\fn$ co-prime to $3$ has a unique generator $n \equiv 1 \mod 3$.

The cubic Dirichlet characters on $\Z[\omega]$ are given by the cubic residue symbols. For each prime $\pi \in \ring$ with $\pi \equiv 1 \mod 3$, there are two primitive characters of conductor $\pi$; the cubic residue symbol $\chi_\pi(a)$ satisfying
$$\chi_\pi(a) =  \( \frac{a}{\pi}\)_3 \equiv a^{{(\N(\pi)-1)/3}} \mod \pi,$$
and its conjugate $\overline{\chi}_\pi= \chi_\pi^2$.
In general, for  $n\in\ring$ with $n \equiv 1 \mod 3$, the cubic residue symbol $\chi_n$ is defined multiplicatively using the characters of prime conductor by
\eqs{\chi_n (a) = \left( \frac{a}{n} \right)_3 = \prod_{\pi \mid n} \chi_\pi (a)^{v_\pi(n)} .}

Such a character $\chi_n$ is primitive when it is a product of characters of distinct prime conductors, i.e. either $\chi_\pi$ or $\chi_{\pi}^2 = \chi_{\pi^2}$.
Moreover, $\chi_n$ is a cubic Hecke character of conductor $n \ring$ if $\chi_n(\omega) = 1$. 
Since
\eqs{\( \frac{\omega}{n}\)_3 = \prod_{\pi \mid n} \omega^{v_\pi(n) (N(\pi)-1)/3} = \omega^{\sum_{\pi \mid n} v_\pi(n) (N(\pi)-1)/3} = \omega^{(N(n) - 1)/3},}
we conclude that a given Dirichlet character $\chi$ is a primitive cubic Hecke character of conductor $n_1 n_2\ring$, co-prime to $3$, provided that $\chi = \chi_n$, where 
\begin{enumerate}
\item $n=n_1 n_2^2$, where $n_1, n_2$ are square-free and co-prime, and
\item  $\N(n) \equiv 1 \mod 9$, or equivalently, $\N(n_1) \equiv \N(n_2) \mod 9$.
\end{enumerate}
\begin{rem}
Given $n\in \ring$, co-prime to 3, write $n=n_1n_2^2 n_3^3$, where $n_1n_2$ is square-free, and $n_1, n_2$ are co-prime. Then, the character $\chi_n$ modulo $n\ring$ is induced by the primitive character $\chi_{n_1} \chi_{n_2^2}$ of conductor $n_1n_2\ring$ unless $n$ is a cube; that is, $n_1, n_2$ are units. 
\end{rem}

We recall the cubic reciprocity for cubic characters.
\begin{lem} Let $m,n \in \ring, m,n \equiv \pm 1 \mod 3.$ Then,
$$
\( \frac{m}{n}\)_3 = \( \frac{n}{m}\)_3.
$$
\end{lem}

Let $\chi$ be a primitive cubic Hecke character to some modulus $\fm = m\ring$, co-prime to 3. The completed Hecke $L$-series is then defined by
\eqs{\Lambda (s,\chi) = (|d_K| N(m))^{s/2} (2\pi)^{-s} \Gamma(s) L(s,\chi),}
where $d_K=-3$ is the discriminant of $K$.
\begin{prop}[{\cite[VII. Cor. 8.6]{Neukirch}}]
The completed $L$-series above is entire, provided $\chi$ is primitive and $\fm \neq \ring$. Futhermore, it satisfies the functional equation
\eqs{\Lambda(s,\chi) = W(\chi) (\N\fm)^{-1/2} \Lambda(1-s,\overline{\chi})}
where 
\eqn{\label{GaussSumofL(s,chi)}
W(\chi)  = \su{x \mod \fm \\ (x,\fm)=1} \chi(x) \e \bigl(\tr(x/m\sqrt{-3})\bigr), }
where $x$ varies over a system of representatives of $\(\ring/\fm\)^\times$. 
\end{prop}

\subsection{Explicit Formula and averaging over the family}  \label{reduce-to-SF}

We now state the explicit formula for cubic L-functions, which relate sums over the zeroes to sums over the coefficients of the $L$-functions. Averaging over the family, we get the main term for the one-level density, and our results are obtained by bounding the error term.

We state but not prove the next two lemmas as their proofs are standard.
\begin{lem}  
Let $\chi$ be a primitive cubic Hecke character modulo $c\ring$. Then, uniformly for $s=\sigma + it$ satisfying $|t|\ge 1$ and $-1 \le \sigma  \le 2$, or that $1/2 < \sigma \le 2$, 
\eqn{\label{eq:L'/Lbound1}
\frac{L'}{L}(s, \chi) = \sum_{|\gamma - t| \le 1} \frac{1}{s-\rho} + O \bigl( \log \bigl (\N(c)(3+|t|)\bigr) \bigr)}
where the sum runs over the zeros $\rho = \beta +i\gamma$ of $\Lambda(s,\chi)$ counted with multiplicity.
\end{lem}
\begin{lem} \label{lem:T1}
With the character as in the previous lemma, for any $T\ge 1$, there is some $T_1 \in [T,T+1]$ such that 
\eqn{\label{eq:L'/Lbound2}
\frac{L'}{L}(\sigma \pm i T_1, \chi) \ll \log^2 \bigl( \N(c)(3+T) \bigr) }
uniformly for $-1\le \sigma \le 2$.
\end{lem}

\begin{lem}[Explicit Formula] \label{lem:ExplicitFormula}
Let $\chi$ be a primitive cubic Hecke character modulo $n\ring$ with $(n,3)=1$, and let $\phi(x)$ be an even function of Schwartz class on $\R$
whose Fourier transform $\phat(y)$ has compact support in $(-v,v)$. Then, 
\mult{\sum_{\rho} \phi \Bigl( \frac{(\rho - 1/2)\log X}{2 \pi i} \Bigr)  =
\phat(0) \frac{\log \N(n)}{\log X} + O \Bigl( \frac 1 {\log X} \Bigr)  \\
- \sum_{\fp} \sum_{1 \le k \le 2} \(\chi (\fp^k) + \chi (\fp^{2k})\)
 \phat \left( \frac{k \log \N\fp}{\log X} \right)
\frac{\log \N\fp}{(\N\fp)^{k/2} \log X}, }
where the implied constant depends only on $\phi$. 
\end{lem}
\begin{proof}
Note that 
\eqs{G(s) := \phi \Bigl( \Bigl( s - \frac{1}{2} \Bigr) \frac{\log X}{2\pi i} \Bigr),}
is holomorphic in $-1 \le \re(s) \le 2$ and satisfies
\eqn{ \label{eq:G(s)}
G(s) = G(1-s), \qquad s^2 G(s) \ll 1.}

Let $T$ be a large real number and $T_1 \in [T, T+1]$ be as in Lemma \ref{lem:T1}. Let $R$ be the rectangle with vertices $2-i T_1, 2+iT_1, -1+iT_1, -1-iT_1$. By Cauchy's Residue Theorem we obtain
\eqs{\sum_{\rho} G(\rho) = \frac{1}{2\pi i} \int_R G(s) 
\frac{\Lambda'}{\Lambda} (s, \chi) ds,}
where the integral is taken counter-clockwise around $R$. By Lemma \ref{lem:T1} and \eqref{eq:G(s)}, the contribution of the horizontal integrals is $\ll T^{-2} \log^2 (T\N(c))$.  Thus, taking limit as $T\to \infty$ and using the functional equations for $G(s)$ and $\Lambda (s, \chi)$, we obtain
\bs{ 
\sum_{\rho} G(\rho) &= 
\frac{1}{2\pi i} \int_{\sigma=2} G(s) 
\biggl(\frac{\Lambda'}{\Lambda} (s, \chi) + \frac{\Lambda'}{\Lambda} (s, \overline{\chi}) \biggr) ds \\
&= \frac{1}{2\pi i} \int_{\sigma=2} G(s) \biggl( \frac{L'}{L}(s, \chi) + \frac{L'}{L}(s, \overline{\chi}) + 2\frac{\Gamma'}{\Gamma} (s) + \log \frac{3\N(c)}{4\pi^2} \biggr) ds.}
Using \eqref{eq:G(s)} again we can shift the contour to $\sigma=1/2$ and conclude that
\eqs{\sum_{\rho} G(\rho) = F(1) \log \frac{3\N(c)}{4\pi^2} + \frac{1}{2\pi i} \int_{(1/2)}
2 G(s) \frac{\Gamma'}{\Gamma} (s) ds - \sum_\fp H(p),}
where
\eqs{H(\fp) = \sum_{n \ge 1} \bigl( \chi (\fp^n) +
\overline{\chi} (\fp^n) \bigr) F(\N(\fp^n)) \log \N\fp}
and
\eqs{F(y) = \frac{1}{2\pi i} \int_{\sigma=1/2} G(s) y^{-s} ds = \phat \Bigl( \frac{\log y}{\log X} \Bigr) / (\sqrt y \log X).}

By the approximate formula (cf.~\cite[8.363.3]{GradRyzhik})
\eqs{\frac{\Gamma'}{\Gamma}\, (a + ib) + \frac{\Gamma'}{\Gamma}\,
(a-ib) = 2 \,\frac{\Gamma'}{\Gamma} (a) + O \left( (b/a)^2
\right)}
we obtain
\bs{\frac{1}{\pi i} \int_{\sigma=1/2} G(s) \frac{\Gamma'}{\Gamma}(s) ds
&= \frac 2 {\log X} \int_{-\infty}^{\infty} \phi(t) \frac{\Gamma'}{\Gamma}  \(\frac{1}{2} + i \frac{2\pi t}{\log X}\)
 dt\\
&= \frac{2 (\Gamma'/\Gamma) (1/2)}{\log X} \phat(0) + O
\((\log X)^{-3} \).
}
Finally, noting that 
\eqs{\sum_{\fp} \sum_{k>2} \left( \chi (\fp^k) + \chi (\fp^{2k})
\right) \phat \left( \frac{k\log \N\fp}{\log X} \right)
\frac{\log \N\fp}{(\N\fp)^{k/2}} \ll \sum_p p^{-3/2} \log p \ll 1}
proves the claimed result.

\end{proof}

Let $\sF$ be one of the families that will be defined in the next two sections. Recall that
\eqs{\cD(X;\phi, \sF) = \frac 1 {\sA_\sF (X)} \sum_{\chi \in \sF} w\left( \frac{\N(\cond(\chi))}{X} \right)  \su{\gamma \\ L(1/2+i\gamma, \chi) = 0}\phi \Bigl( \frac{\gamma \log X}{2 \pi} \Bigr),}
where $\N(\cond(\chi))$ is the norm of the conductor of the primitive character $\chi$ and
\eqs{\sA_\sF (X) = \sum_{\chi \in \sF} w\left( \frac{\N(\cond(\chi))}{X} \right).}
Using the explicit formula, Lemma \ref{lem:ExplicitFormula}, we obtain
\bsc{\label{after-EF}
\cD(X; \phi, \sF)  &= \frac{\phat(0) }{\sA_\sF (X) \log X} \sum_{\chi \in \sF} w\left(\frac{\N ( \cond(\chi))}{X} \right) \log \N( \cond(\chi) ) \\
&\quad + O \Bigl( \frac {X^{v/2} + |\cE_\sF (X)|}{\sA_\sF (X)\log X} \Bigr), }
where
\eqn{\label{EF(X)}
\cE_\sF  (X) = \sum_{\chi \in \sF\cup \{ 1\}} w\left( \frac{\N (\cond(\chi))}{X} \right) 
\sum_{\fp \nmid 3} \sum_{1 \le k \le 2} \chi (\fp) \phat \left( \frac{k \log \N\fp}{\log X} \right)
\frac{\log \N\fp}{(\N\fp)^{k/2}}.}
To establish Theorems \ref{thm-main} and \ref{thm-trivial}, we need to find the largest $v$ such that $\ef (X) = o(\sA_\sF (X) \log X)$ for each family in question. Thus, the rest of the paper will be devoted to the estimate of the sum
\eqn{\label{SF(y)}
S_\sF (y) = \sum_{\chi \in \sF\cup \{ 1\}} w(\cond(\chi)/X) \su{\N\fp \le y\\ \fp \nmid 3}  \chi (\fp) \log \N\fp.
}

\subsection{Cubic Gauss sums and Poisson Summation for cubic characters}

We now define the generalized Gauss sums associated with the cubic residue symbols $\chi_n$, where $n \equiv 1 \mod 3 \in \ring$. Notice that we do not suppose that $n$ is square-free, and  these characters are not necessarily primitive.
Let
\eqn{\label{ShiftedGaussSum}
g(r,n) = \sum_{\alpha \mod n} \chi_n (\alpha) \e \bigl(\tr( r \alpha/n)\bigr).}
Then, for $(n, \sqrt{-3})=1$, 
\eqs{W(\chi_{n}) = \chi_n(\sqrt{-3}) g(1,n),}
where $W(\chi)$ is the sign of the functional equation given by \eqref{GaussSumofL(s,chi)}.

The following two lemmas are classical results about cubic Gauss sums which can be found in \cite{HBP1979}, or easily checked, and we include them without proof.
\begin{lem} \label{cubic-GS-lemma1} 
Let $n, n_1, n_2 \equiv 1 \mod 3$ and $s, r$ be elements of $\Z[\omega]$. \\ 
If $(s, n)=1$, 
\eqs{ 
g(rs, n) = \overline{\chi_{n}}(s) g(r,n).} 
If $(n_1, n_2)=1$, 
\eqs{ 
g(r, n_1 n_2) = \overline{\chi_{n_2}}(n_1) g(r,n_1) g(r, n_2) = g(rn_1, n_2) g(r, n_1).}
\end{lem}
\begin{lem} \label{cubic-GS-lemma2}
Let $\pi \equiv 1 \mod 3$ be a prime,  $(\pi, r)=1$, where $r\equiv 1 \mod 3$. Let $k, j$ be integers with $k > 0$ and $j \ge 0$.

If $k=j+1$, 
\eqs{g(r\pi^j,\pi^k) = \N(\pi^j) \times \ar{-1 & 3\mid k\\[1mm] g(r,\pi) & k\equiv 1 \mod 3 \\[1mm] \overline{g(r,\pi)} & k \equiv 2 \mod 3.  }}

If $k \neq j+1$,
\eqs{g(r\pi^j,\pi^k) = \begin{cases}
\varphi_K(\pi^k)& \text{if } 3\mid k, \, k \le j \\ 0 & \text{otherwise.}  
\end{cases}}
\end{lem}

Our next lemma is obtained from the Poisson summation formula over $\Z^2$, which is essentially \cite[Lemma 10]{HB2000}. 
\begin{lem} \label{lem:Poisson} 
Let $\chi$ be a primitive character of $(\ring/\ff \ring)^\times$. Then,
\eqs{\sum_{n \in \ring} \chi (n) w(\N(n)/Y) = \frac Y {W(\overline{\chi})} \sum_{n \in \ring} \overline{\chi} (n) \widehat{w}( \sqrt{ Y\N(n)/\N( \ff)}), }
where $W(\chi)$ is given by \eqref{GaussSumofL(s,chi)} and for $t \in \R, t \ge 0$,
\begin{align*}
\widehat {w} (t) &= \int_\R \int_\R w(\N(x+y\omega)) \e \bigl(\tr( t (x+y\omega)/\sqrt{-3})\bigr) dx dy, \\ 
&= \int_\R \int_\R w(\N(x+y\omega)) \e (- ty) dx dy. 
\end{align*}
\end{lem}

We will often use the following lemma, which makes it easier to keep track of the size of various parameters when optimizing an estimate of the
form
$
U \ll AH^a + BH^{-b},$ where $A, B, a, b$ are positive constants and $H$ can be chosen. Simply choosing $H$ to satisfy $AH^a = BH^{-b}$ leads to $U \ll (A^b B^a)^{1/(a+b)},$ which is the best bound apart for the value of the other parameters in the implied constant. This can be generalized to the following lemma.
\begin{lem}[{\cite[Lemma 2.4]{GraKol}}] \label{lemma-kolesnik} 
Suppose that $$L(H) = \sum_{i=1}^m A_i H^{a_i} + \sum_{j=1}^n B_j H^{-b_j},$$ where
$A_i, B_j, a_i, b_j$ are positive. Suppose that $H_1 \le H_2$. Then, there is some $H$ with $H_1 \le H \le H_2$ such that
$$
L(H) \ll \sum_{i=1}^m \sum_{j=1}^n \left(A_i^{b_j} B_j^{a_i}\right)^{1/(a_i+b_j)} + \sum_{i=1}^m A_i H_1^{a_i} + \sum_{j=1}^n B_j H_2^{-b_j},
$$
where the implied constant depends only on $m$ and $n$.
\end{lem}

\section{Proof of Theorem \ref{thm-trivial}}
\label{full-family}
Let $\sF_3$ be the family of cubic residue symbols $\chi_{ab^2}$ where $a,b  \equiv 1 \mod 3\in \Z[\omega]$ are square-free and co-prime, and 
$ab^2 \equiv 1 \mod 9$.
By a slight abuse of notation (dropping the letter $\chi$), we write
\eqs{
	\sF_3 = \Big\{ ab^2\in \ring\setminus\{1\} \;:\;  \begin{array}{l}
a, b \equiv 1 \mod 3 \text{ both square-free},\\
(a,b)=1, \;ab^2 \equiv 1 \mod 9
\end{array} \Big\}. 
}

\begin{lem} \label{lem:sizeofwholefamily} 
Let $\phi$ be a Schwartz class test function whose Fourier transform is supported in $(-v,v)$. Then, for the family $\sF_3$ defined above, 
\bs{&\sA_{\sF_3}  (X) = \frac{2\pi\fw(1)}{9h_{(9)}\sqrt{3}}  \prod_{\fp \nmid 3} \Bigl( 1 - \frac 3 {\N\fp^2} + \frac 2 {\N\fp^3}\Bigr) X\log X + O(X), \\
&\sum_{ab^2 \in \sF_3} w(\N(ab)/X) \log \N(ab)  = \sA_{\sF_3}(X) \log X + O(X \log X),}
and for any $\eps>0$,
\eqs{\ef (X) \ll 
\begin{cases}
X^{1/2+v/2+\eps} & \text{under GRH}\\
X^{1/2+v+\eps} & \text{unconditionally,}
\end{cases}}
where $\ef (X)$ is defined in \eqref{EF(X)}, $\fw (s)= \int _0^\infty w(x) x^{s-1} ds$ is the Mellin transform of $w$ and $h_{(9)} = |J^{(9)}/P^{(9)}|$ is the order of the ray class group modulo $9$. 
\end{lem}

\begin{proof} 
We first write 
\eqs{\sA_{\sF_3}  (X) + w(1/X) = \frac 1 {h_{(9)}} \sum_{\psi \mod 9} \su{q \equiv 1 \mod 3\\ q \SF} w(\N(q)/X) \psi(q) \su{b \equiv 1 \mod 3\\ b \mid q} \psi(b),}
where $\psi$ runs over the ray class characters modulo 9. By Mellin inversion,
\eqs{\sA_{\sF_3}  (X) + w(1/X) = \frac 1 {h_{(9)}} \sum_{\psi \mod 9} \frac{1}{2\pi i} \int_{\re(s)=2} \fw (s) X^s G_\psi(s)
ds,}
with generating series
\bsc{ \label{G_psi}
G_\psi(s) &=  \su{q \in \ring\\ q \equiv 1 \mod 3\\ q \;\text{square-free}} \frac{\psi(q)}{\N(q)^s} \su{b \in \ring \\ b \equiv 1 \mod 3\\ b \mid q} \psi(b) \\
&= \pr{\pi \equiv 1 \mod 3} \Bigl( 1 + \frac{\psi(\pi)}{\N(\pi)^s} \bigl(1+\psi(\pi)\bigr)\Bigr) = L(s,\psi) L(s,\psi^2) F(s,\psi),}
where
\eqs{L(s,\psi) = \pr{\pi \equiv 1 \mod 3} \Bigl( 1 - \frac{\psi(\pi)}{\N(\pi)^s} \Bigr)^{-1} = \sum_\fa \frac{\psi(\fa)}{\N\fa^s}}
is the Hecke L-function associated with the character $\psi$,
and
\bs{
F_\psi(s) &= \pr{\pi \equiv 1 \mod 3} \Bigl( 1 - \frac{\psi(\pi)}{\N(\pi)^s} \Bigr) \Bigl( 1 - \frac{\psi^2(\pi)}{\N(\pi)^s} \Bigr) \Bigl( 1 + \frac{\psi(\pi)}{\N(\pi)^s} \bigl(1+\psi(\pi)\bigr)\Bigr) \\
&=\pr{\pi \equiv 1 \mod 3} \Bigl( 1 - \frac{\psi^2 (\pi) + \psi^3 (\pi)+\psi^4(\pi)}{\N(\pi)^{2s}} + \frac{\psi^4(\pi)+\psi^5(\pi)}{\N(\pi)^{3s}}\Bigr)
}
Note that $F_\psi(s)$ converges absolutely for $\re(s) > 1/2$ for all $\psi$ modulo 9,  whereas $L(s,\psi)$ and $L(s,\psi^2)$ both have analytic continuation to entire functions except for a simple pole at $s=1$ when the characters are principal. 

Hence, moving the contour to $\re(s)=1/2+\eps$, we get
\multn{ \frac{1}{2\pi i} \int_{\re(s)=2} \fw (s) X^s G_\psi (s) ds \\ \label{after-cauchy}
= \sum_{\substack{\psi=\psi_0 \\  \psi^2 = \psi_0}} \mathop{\res}_{s=1} \,\bigl( X^s \fw (s) L(s,\psi) L(s,\psi^2) F_\psi(s)\bigr) + O \left( X^{1/2+\eps} \right),}
where, for the error term, we used the fact that $\fw(s) \ll |s|^{-n}$ for all $n \ge 1$, with the classical convexity bound
$ L(s,\psi), L(s,\psi^2) \ll_\psi |s|^{\max(0,1+\eps-\re(s))}$
for $0 \le \re(s) \le 2$, uniformly for any $\eps >0$. 

When $\psi = \psi_0$, there is a pole of order 2 at $s=1$ with residue
\eqs{\left( \frac 4 9 (\mathop{\res}_{s=1} \zeta_K (s))^2 F_{\psi_0}(1) \fw (1) \right) \, X\log X + O(X),}
and when $\psi^2 =\psi_0, \psi \neq \psi_0$, then there is a simple pole and the contribution of the residue is $O(X)$. Substituting in \eqref{after-cauchy}, this proves the first result.

For the second assertion, using Mellin inversion, we have
\eqs{\sum_{ab^2 \in {\sF_3}  \cup \{ 1 \}} w(\N(ab)/X) \log \N(ab)  = -\frac{1}{2\pi i h_{(9)}} \sum_{\psi \mod 9}  \int_{\re(s)=2} \fw (s) X^s G_\psi'(s) ds,}
where $G_\psi(s)$ was defined above in \eqref{G_psi}.  Using integration by parts, for each character $\psi$, the above integral is
\eqs{-\int_{\re(s)=2} \fw' (s) X^s G_\psi(s) ds -  \log X \int_{\re(s)=2} \fw (s) X^s  G_\psi(s) ds.}
Working as above, the main contribution from each integral comes from the double pole of $G_{\psi_0}(s)$ at $s=1$. For the first integral, we bound this contribution by
$O(X\log X)$. Summing the second integral over the characters gives $\sA_{\sF_3} (X) + w(1/X)$. This proves the second assertion.

Finally, we prove the last assertion.
Let
\eqs{S_{\sF_3}  (y) = \sum_{3< \N\fp \le y} \log \N\fp \sum_{ab^2 \in {\sF_3} \cup \{1\}} \chi_{ab^2} (\fp) w(\N(ab)/X).}

Writing each prime ideal $\fp$ as $\fp = ( \pi )$ with $\pi \equiv 1 \mod 3$, we have
\bs{S_{\sF_3}  (y) &= \frac{1}{h_{(9)}} \sum_{\psi \mod 9}  \su{\pi \equiv 1\mod 3\\ \N(\pi)\le y} \log{\N(\pi)} \su{q \in \ring \\q \equiv 1\mod 3\\q \SF} (\chi_\pi \psi)(q) w(\N(q)/X) \su{b \in \ring \\b\equiv 1\mod 3\\b \mid q} (\chi_\pi \psi)(b) \\
	&= \frac{1}{ h_{(9)}} \sum_{\psi \mod 9}  \su{\pi \equiv 1\mod 3\\ \N(\pi)\le y} \log{N(\pi)} \;\frac{1}{2 \pi i}\int_{(2)} X^s \fw(s) G_{(\chi_\pi \psi)} (s) ds.
}
We evaluate the integral working as above.
Since $(\psi \chi_\pi)$ is non-trivial for every character $\psi \mod 9$, the generating series has no pole when $\re(s) >1/2$. 
We move the integral to $\re(s)=1/2 + \eps$, and we use the bound
\begin{align} \label{bound-L-at-onehalf}
L(\textstyle \frac12 + \varepsilon+it, \chi) \ll \begin{cases}  t^\eps \; \N (\cond(\chi))^\eps & \mbox{under GRH} \\
|t|^{1/2+\eps} \; \N(\cond(\chi))^{1/4} &\mbox{unconditionally,}
\end{cases}
\end{align}
which holds for any non-trivial character $\chi$. 
This gives
\eqs{S_{\sF_3}  (y) \ll  \begin{cases} X^{1/2+\eps} y^{1+2\eps} & \text{under GRH}
\\ X^{1/2+\eps} y^{3/2} & \text{unconditionally.}\end{cases}}
By partial integration
\bs{
\ef(X) &= \sum_{1 \le k \le 2} \int_3^\infty \phat \Bigl( \frac{k\log y }{\log X} \Bigr) y^{-k/2} dS_{\sF_3} (y) \\
&= \sum_{1 \le k \le 2} \int_3^{X^{v/k}} S_{\sF_3} (y) y^{-k/2-1} \Bigl( \frac k 2 \phat \Bigl(\frac{k\log y }{\log X} \Bigr) - \phat' \Bigl( \frac{k\log y }{\log X} \Bigr) \frac k {\log X}\Bigr) dy \\
&\ll 
\begin{cases}
X^{1/2+v/2+\eps} & \text{under GRH}\\
X^{1/2+v+\eps} & \text{unconditionally.}
\end{cases}  
}
This establishes the third assertion.
\end{proof}
Using the lemma in \eqref{after-EF} of Section \ref{reduce-to-SF}, the proof of Theorem \ref{thm-non-GRH} for the family $\sF_3$ follows, and assuming GRH, the proof of Theorem \ref{thm-trivial} follows.

\section{Proof of Theorem \ref{thm-main}}
\label{section-thm-main}

Let $\sF_3'$ be the family of primitive cubic Dirichlet characters determined by the cubic residue symbols $\chi_n$, where $n \neq 1$ is square-free and congruent to 1 modulo 9. 
Again, with a slight abuse of notation, we write \eqn{\label{thinfamily}
\sF_3' =\{ n\in \Z[\omega] : n\neq 1, n \equiv 1 \mod 9 \text{ and square-free} \}.}

\begin{lem} \label{lem:Aw(X)asymptotic}
Let $w:\R \to (0,\infty)$ be an even Schwartz function. Then, 
\eqs{\sA_{\sF_3'} =  \zeta_K^{-1}(2) \frac{\pi \fw(1) }{4\sqrt{3}h_{(9)}} X + O(X^{1/2+\eps}),}
\eqs{\sum_{n \in \sF_3'} w(\N(n)/X) \log \N(n) = \sA_{\sF_3'} \log X + O(X),}
and for any $\eps>0$,
\eqs{\eft (X) \ll 
\begin{cases}
X^{1/2+v/2+\eps} & \text{under GRH}\\
X^{1/2+3v/4+\eps} & \text{unconditionally,} 
\end{cases}}
where $\fw (s)= \int _0^\infty w(x) x^{s-1} ds$ is the Mellin transform of $w$,  $h_{(9)} = |J^{(9)}/P^{(9)}|$ is the order of  the ray class group modulo $9$, and $\zeta_K (s)$ is the Dedekind zeta function of $K$.

\end{lem}
\begin{proof} 
As in the proof of Lemma \ref{lem:sizeofwholefamily}, we start by writing 
\begin{align*}
\sA_{\sF_3'} = \frac 1 {2\pi i h_{(9)}} \sum_{\psi \mod 9} \int_{\sigma =2} X^s \fw (s) G_\psi (s) ds - w(1/X), \end{align*}
where
\begin{equation} \label{G-thinfamily} G_\psi (s) = \su{\fa \in J^{(9)}} \frac{\psi(\fa)|\mu{(\fa)|}}{(\N\fa)^s} = \left\{
\def\arraystretch{1.2}
\begin{array}{l@{\quad\text{if }}l}
\zeta^{-1}_K (2s)\zeta_K (s) ( 1+ 3^{-s})^{-1}
& \psi = \psi_0\\
F(s) L(s,\psi)& \psi \neq \psi_0,
\end{array} \right.  
\end{equation}
with $\psi_0$ the principal character modulo 9, and
\eqs{F(s) = \sum_\fa \frac{\mu(\fa)\psi^2(\fa)}{(\N\fa)^{2s}}, \qquad L(s, \psi) = \sum_{\fa} \psi(\fa) (\N\fa)^{-s}.}
Proceeding as before, shifting the contour to $\re(s) = 1/2 +\eps$ and using the convexity bound we conclude that
\eqs{\sA_{\sF_3'} = \frac 1 {h_{(9)}} \mathop{\res}_{s=1} \( X^s \fw (s) G_{\psi_0} (s) \) + O(X^{1/2+\eps}).}
Since $G_{\psi_0}(s)$ has a simple pole at $s=1$, we have
\bs{\mathop{\res}_{s=1} \( X^s \fw (s) L_{\psi_0} (s) \) = X \fw(1) \lim_{s\to 1} (s-1) G_{\psi_0} (s) = X \fw(1) \zeta_K^{-1}(2) \frac{\pi}{4\sqrt{3}},
}
where we used (cf. \cite[Ch VII. Corollary 5.11]{Neukirch}) 
\eqs{\mathop{\res}_{s=1} \zeta_K(s) = \frac{2\pi h_K R}{6|d_K|^{1/2}} = \frac{\pi}{3\sqrt{3}}.}
This gives the first claim. The second identity follows as in the previous section along the same lines using integration by parts. For the third identity, using the bound \eqref{bound-L-at-onehalf} and working as in the proof of Lemma \ref{lem:sizeofwholefamily} with the generating series \eqref{G-thinfamily}, we get
\eqs{S_{\sF_3'} (y) \ll  \begin{cases} X^{1/2+\eps} y^{1+\eps} & \text{under GRH}
\\ X^{1/2+\eps} y^{5/4} & \text{unconditionally.}\end{cases}}
The third identity follows by partial integration.
\end{proof}

Using Lemma \ref{lem:Aw(X)asymptotic} in equation \eqref{after-EF} of Section \ref{reduce-to-SF}, the proof of Theorem \ref{thm-non-GRH} for the family $\sF_3'$ follows. 
We also remark that using the bound for $\eft(X)$ of the lemma, under GRH, gives the one-level density for the family $\sF_3'$ only when the support of the Fourier transform is contained in $(-1,1)$. We now 
turn to the proof of Theorem \ref{thm-main} which increases this support.

Using Lemma \ref{lem:Aw(X)asymptotic} in \eqref{after-EF}, we have that
\eqs{\cD(X; \phi, \sF_3')  = \phat(0)+ O \Bigl( \frac{1}{\log{X}} \Bigr) + O \Bigl( \frac {X^{v/2} + |\eft (X)|}{X \log X} \Bigr), }
and we want to show that $\eft (X) = \underline{o} ( X \log X )$ when $\text{supp}(\phat) \subseteq (-v,v)$. For the family $\sF_3'$ in \eqref{thinfamily}, the sum defined in \eqref{SF(y)} can be written as 
\bs{S_{\sF_3'}  (y) &= \sum_{3< \N\fp \le y} \log \N\fp \sum_{n \in {\sF_3'} \cup \{1\}} \chi_n (\fp) w(\N(n)/X) \\
 &= \sum_{\substack{n \in \ring \\ n \equiv 1 \mod 9}} w(\N(n)/X) \sum_{3< \N\fp \le y} \chi_n (\fp) \log \N\fp \su{d \in \ring\\ d \equiv 1 \mod 3 \\ d^2 \mid n } \mu_K(d) \\
&= \sum_{\substack{ d \in \ring \\ d \equiv 1 \mod 3}} \mu_K(d) \sum_{\substack{ n \in \ring \\ n \equiv d^{-2} \mod 9}} w ( {\N(nd^2)}/{X} ) \sum_{3 < \N\fp \le y} \chi_{nd^2} (\fp) \log \N\fp.
}
where $\mu_K$ is Moebius function over $K$ and we use the detector
\eqs{\su{ d \in \ring \\ d \equiv 1 \mod 3 \\ d^2 \mid n} \mu_K(d) = \begin{cases}
1&\text{if $n$ is square-free}\\
0&\text{otherwise.}
\end{cases}}
We write $S_{\sF_3'} (y) =S_1(y) + S_2(y)$ where
\begin{align} \label{def-S1} S_1 (y) &= \sum_{3 < \N\fp \le y} \log \N\fp \su{d \in \ring \\d \equiv 1 \mod 3\\ \N(d) \le D} \mu_K(d) \chi_{d^2} (\fp)  \sum_{\substack{ n \in \ring \\n \equiv d^{-2} \mod 9}} w  \Bigl(\frac{\N(nd^2)}{X}\Bigr) \chi_n (\fp) \\ \label{def-S2}
S_2 (y) &= \su{d \in \ring \\d \equiv 1 \mod 3\\ \N(d) > D} \mu_K(d)\su{n \in \ring \\ n \equiv d^{-2} \mod 9} w(\N(nd^2)/X) \sum_{3 < \N\fp \le y} \chi_{nd^2} (\fp) \log \N\fp.
\end{align}

\subsection{Estimate of $S_2(y)$}

\begin{lem} \label{lem:CharSumwithGRH}
Given $n \equiv 1 \mod 3$ in $\ring$ with $\N(n) \equiv 1\mod  9$, write $n= n_1n_2^2n_3^3$ with $n_1, n_2$ square-free and co-prime, and $n_i \equiv 1\mod 3$. Assuming GRH for $L(s,\chi_n)$, and that $n$ is not a cube, the estimate  
\eqs{\sum_{\N\fp \le x} \chi_n (\fp) \log \N\fp
\ll x^{1/2} \log^3 \(x\N(n)\)}
holds for $x > 1$. \end{lem}
\begin{proof} We shall give only a sketch of the proof as it uses standard techniques.

The estimate trivially holds for $x < 3$ since the sum equals zero and the right side is positive. Thus we assume $x\ge 3$, and write 
\eqs{L(s,\chi_n) = \prod_{\fp \nmid n} \bigl(1 - \chi_n (\fp) \N\fp ^{-s} \bigr)^{-1} = \prod_{\fp \mid d} \bigl(1 - \chi (\fp) \N\fp ^{-s} \bigr) L(s,\chi),
}
where $\chi = \chi_{n_1 n_2^2}$ is the primitive character to modulus $n_1n_2\ring$ that induces $\chi_n$, and where $d$ denotes the product of primes dividing $n_3$, but not $n_1n_2$.  For $T>1$, we find $T_1\in [T,T+1]$ given by Lemma \ref{lem:T1} for $L(s,\chi)$. Using Perron's formula with $a=1+ (2\log x)^{-1}$,
\bs{
\sum_{\N\fp \le x} \chi_n(\fp) \log \N\fp  &= \frac{1}{2\pi i} \int_{a - i T_1}^{a + i T_1} -\frac{L'}{L} (s, \chi_n)
\frac{x^s}{s} ds \\
&\qquad + O \biggl( x^{1/2} + \frac{x \log x} T + \log x + \frac{x\log^2 x}{T} \biggr).
}
Next we shift the contour to $b= 1/2 + (2\log x)^{-1}$ and use \eqref{eq:L'/Lbound1} and \eqref{eq:L'/Lbound2} to estimate the horizontal and vertical integrals to get
\mult{
\sum_{\N\fp \le x} \chi_n(\fp) \log \N\fp \ll \frac{x\log^2 (\N(n_1n_2)(3+T))}{T\log x}  + \Bigl(x^{1/2} \log T + \frac{x}{T\log x} \Bigr) \log \N(d)\\
+ \frac{x\log^2 x}{T} + x^{1/2} \(\log x \log (3\N(n_1n_2)) + \log T \log^2 (\N(n_1n_2)(3+T))\) .}
Choosing $T = x$ gives the claimed estimate. \end{proof}
\begin{lem}
Let $1 < D \le \sqrt X$. Then,
\eqn{\label{eq:S2bound}
S_2(y) \ll \begin{cases}  \cfrac{X y^{1/2}}{D} \log^3{(yX)}   + y \min \Big\{ \cfrac{X}{D^2}, X^{1/3} \log X \Big\} 
& \text{under GRH} \\    \\
\displaystyle\frac{y X}{D}  & \text{unconditionally} \end{cases} }
\end{lem}
\begin{proof} 
Trivially estimating the innermost sum over primes in \eqref{def-S2} gives
\eqs{S_2 (y) \ll y \su{d \in \ring \\ d \equiv 1 \mod 3\\ \N(d) > D} \mu_K(d)\su{n \in \ring \\ n \equiv d^{-2} \mod 9} w(\N(nd^2)/X).}
Since $1 \le D \le \sqrt X$, we have 
\bs{\su{d \in \ring \\ D < \N(d) \le \sqrt X} \biggl(\su{n \in \ring \\ \N(n) \le X/\N(d^2)} + \su{n \in \ring \\ \N(n) > X/\N(d^2)}\biggr) w(\N(nd^2)/X) \ll  X/D, } 
and
\eqs{\su{d \in \ring \\ d \equiv 1 \mod 3\\ \N(d) > \sqrt X} \sum_{n \in \ring} w(\N(nd^2)/X) \ll X^{1/2} \le X/D.}

If we assume GRH, we can use Lemma \ref{lem:CharSumwithGRH} whenever $nd^2$ is not a cube. The contribution of these terms to $S_2(y)$ is
\eqs{\ll y^{1/2} \su{d \in \ring \\ \N(d) > D} \sum_{n \in \ring} w(\N(nd^2)/X) \log^3 (y\N(nd^2)) \ll \frac {Xy^{1/2}}D \log^3 (yX).}
Note that $nd^2$ is a cube if and only if $n =da^3$ for some $a \in\ring$, since $d$ is square-free. Thus, the contribution of cubes to $S_2(t)$ is 
\eqs{\ll y\su{d \in \ring \\ d \equiv 1 \mod 3\\ \N(d) > D} \sum_{a \in \ring} w(\N(a^3d^3)/X) \ll y\min\{ X/D^2, X^{1/3} \log X\}.}
Combining all the estimates, the result follows. 
\end{proof}

\subsection{Estimate of $S_1(y)$}
Writing each prime ideal $\fp$ co-prime to 3 as  $\fp=\pi\ring$ with $\pi \equiv 1 \mod 3$, and using cubic reciprocity, we have
\eqs{\chi_{nd^2} (\fp) = \chi_{nd^2} (\pi) = \chi_\pi (nd^2) = \chi_\pi (n) \chi_\pi (d^2),}
where the first equality follows since $nd^2 \equiv 1 \mod 9$. Replacing in \eqref{def-S1}, and using the ray class characters to detect the congruence condition, we have
\bs{S_1 (y) 
&= \frac 1 {h_{(9)}} \sum_{\psi \mod 9} \su{\pi \equiv 1 \mod 3\\ \N(\pi) \le y} \log \N(\pi) \su{d \equiv 1 \mod 3\\ \N(d) \le D} \mu_K(d) \chi_{\pi} (d^2) \psi (d^2) \\
&\quad \times \sum_{n\in \ring} w  \Bigl(\frac{\N(nd^2)}{X}\Bigr) \chi_{\pi} (n) \psi (n),
}
where $\psi$ runs over the ray class characters modulo $9$, and $h_{(9)}$ is the order of the ray class group modulo 9.
When $\psi$ is the principal character modulo 9, the innermost sum over $n$ is
\eqs{\sum_{n\in \ring} w  \Bigl(\frac{\N(nd^2)}{X}\Bigr) \chi_{\pi} (n) - \sum_{n\in \ring} w  \Bigl(\frac{\N((1-\omega)nd^2)}{X}\Bigr) \chi_{\pi} ((1-\omega)n).}
Hence, we can write $S_1(y) = S_{11}(y) + S_{12}(y)-S_{13}(y)$, where 
\mult{S_{11} (y) = \frac 1 {h_{(9)}} \su{\psi \mod 9\\ \psi \neq \psi_0} \su{\pi \equiv 1 \mod 3\\ \N(\pi) \le y} \log \N(\pi) \\
\times \su{d \equiv 1 \mod 3\\ \N(d) \le D} \mu_K(d) \chi_{\pi} (d^2) \psi (d^2) \sum_{n\in \ring} w  \Bigl(\frac{\N(nd^2)}{X}\Bigr) (\psi \chi_\pi) (n),
}
\eqs{S_{12} (y) = \frac 1 {h_{(9)}} \su{\pi \equiv 1 \mod 3\\ \N(\pi) \le y} \log \N(\pi) \su{d \equiv 1 \mod 3\\ \N(d) \le D} \mu_K(d) \chi_{\pi} (d^2) \sum_{n\in \ring} w  \Bigl(\frac{\N(nd^2)}{X}\Bigr) \chi_\pi (n),
}
and
\eqs{S_{13} (y) = \frac 1 {h_{(9)}} \su{\pi \equiv 1 \mod 3\\ \N(\pi) \le y} \log \N(\pi) \su{d \equiv 1 \mod 3\\ \N(d) \le D} \mu_K(d) \chi_{\pi} ((1-\omega)d^2) \sum_{n\in \ring} w  \Bigl(\frac{3\N(nd^2)}{X}\Bigr) \chi_\pi (n).
}
Note that since $9$ is a prime power, each non-principal character $\psi$ takes on the same values as the primitive character that induces $\psi$. Hence, we can treat each $\psi$ as primitive, and we denote its conductor by $\ff_\psi$. Then, $\psi \chi_\pi$ is a primitive character of modulus $\ff\ring$, where $\ff=\ff_\chi \pi$, since $(\pi,9)=1$.

We now apply Poisson summation (Lemma \ref{lem:Poisson}) to the innermost sums over $n \in \ring$. For $S_{11} (y)$, this gives
\multn{\label{after-Poisson}
\sum_{n\in \ring} w  (\N(nd^2)/X ) (\psi \chi_\pi) (n) = 
\\ \frac {X\psi(-1) \chi_{\pi}(\ff_\psi \sqrt{-3}) \psi (\pi) W({\psi}) g(1, \pi)}{\N(d^2) \N(\ff_\psi \pi)} \sum_{n \in \ring} \overline{\psi} (n) \overline{\chi}_\pi (n)
\tilde{w} \biggl( \frac{X\N(n)}{ \N( d^2) \N(\ff_\psi \pi)} \biggr),
}
where $\tilde{w}(t) =\widehat{w}(\sqrt{t}),$ and we used the identities
\eqs{W(\overline{\chi_\pi \psi}) = \chi_\pi(-1) \psi(-1) \overline{W(\chi_\pi \psi)} = 
\psi(-1) \overline{W(\chi_\pi \psi)} , \;\; |W({\chi_\pi \psi})|^2 = N(\pi) N(\ff_\psi),}
and (from the Chinese Remainder Theorem),
\eqs{W(\chi_\pi \psi) = \chi_\pi (\ff_\psi) \psi (\pi) W(\chi_\pi) W(\psi) = \chi_\pi (\ff_\psi) \psi (\pi) W(\psi) \chi_\pi(\sqrt{-3}) g(1, \pi)}
with $g(1, \pi)$ defined by \eqref{ShiftedGaussSum}.

We now insert \eqref{after-Poisson} into $S_{11}(y)$, then write each $\tilde n \in \ring$ as $u n (1-\omega)^k$, where $n\equiv 1\mod 3$, $u$ is a unit and $k\ge 0$, and use
\eqs{\chi_\pi(d^2)  \overline{\chi_\pi} (n)  g(1, \pi) = \overline{\chi_\pi} (nd) g(1,\pi) = \begin{cases} g(nd , \pi) & \text{when $(nd, \pi)=1$} \\ 0 & \text{otherwise} \end{cases}}
to get 
\begin{align} \label{this-is-S11}
S_{11}(y) &= \frac X {h_{(9)}} \su{\psi \mod 9\\ \psi \neq 1} \frac{W(\psi)\psi(-1)}{\N(\ff_\psi)} \su{d \equiv 1 \mod 3\\ \N(d) \le D} \frac{\mu_K(d) \psi (d^2) }{\N(d^2)} 
\sum_{u \in \ring^\times} \su{n \in \ring\\ n \equiv 1 \mod 3} \overline{\psi} (un)  \nonumber \\ 
&\quad \times \su{\pi \equiv 1 \mod 3\\ (\pi, nd)=1 \\\N(\pi) \le y} 
\frac {g(nd,\pi)}{\N(\pi)}
\tilde{w} \biggl( \frac{X\N(n)}{ \N( d^2\ff_\psi \pi)} \biggr) 
\chi_{\pi}(u^2 \sqrt{-3}\ff_\psi) \psi (\pi) \log \N(\pi) .
\end{align}

We get similar and simpler (since there is no character $\psi$) formulae for $S_{12}(y)$ and $S_{13}(y)$, namely
\mult{S_{12} (y) = \frac X {h_{(9)}} \su{d \equiv 1 \mod 3\\ \N(d) \le D} \frac{\mu_K(d)}{\N(d^2)} \sum_{\substack{u\in \ring^\times\\k \ge 0}}  \sum_{\substack{n \in \ring\\ n \equiv 1 \mod 3}} \\ \times \su{\pi \equiv 1 \mod 3\\ (\pi, nd)=1 \\\N(\pi) \le y}   \frac {g(nd,\pi)} {\N(\pi)}  \chi_\pi (u^2 (1-\omega)^{2k} \sqrt{-3})  \tilde{w} \biggl( \frac{X3^k\N(n)}{ \N( d^2\pi)} \biggr)  \log \N(\pi), }
and
\mult{S_{13}(y) = \frac X {h_{(9)}} \su{d \equiv 1 \mod 3\\ \N(d) \le D} \frac{\mu_K(d)}{\N(d^2)} 
\sum_{\substack{u\in \ring^\times\\k \ge 0}} 
 \sum_{\substack{n \in \ring\\ n \equiv 1 \mod 3}} \\ \times \su{\pi \equiv 1 \mod 3\\ (\pi, nd)=1\\\N(\pi) \le y}
\frac {g(nd,\pi)} {\N(\pi)}  \chi_{\pi} (u^2 (1-\omega)^{2k+1} \sqrt{-3})
 \tilde{w} \biggl( \frac{X3^{k-1}\N(n)}{\N( d^2\pi)} \biggr)  \log \N(\pi).}

We first estimate the sum over primes. To this end, we define for any character $\lambda$ on $\ring$ and $r \equiv 1 \mod 3$, 
\bsc{\label{H(Z,r,Lambda)}
H(Z,r,\lambda) &= \su{\pi \equiv 1 \mod 3\\ (\pi,r)=1\\ \N(\pi) \le Z}
\tilde g_\lambda (r,\pi) \log \N(\pi)  = \su{c \equiv 1 \mod 3\\ (c,r)=1\\ \N(c) \le Z} \tilde g_\lambda (r,c) \Lambda(c)\\
\tilde g_\lambda (r,c) &= g(r,c) \lambda(c) \N(c)^{-1/2},  }
where the second equality in the first line follows from Lemma \ref{cubic-GS-lemma2}.
\begin{thm} \label{average-GS} 
Let $\lambda$ be a character on $\ring$ and $r  \in \ring$ with $r\equiv 1 \mod 3$. Then for any $\eps > 0$,
\mult{
H(Z,r,\lambda) \ll Z^{\eps} \Bigl( Z^{2/3} \N(r)^{1/6+\eps} + Z^{1/2} \N(r)^{1/4+\eps} \\
+ Z^{5/6} \N(r)^{1/12 +\eps} + Z^{4/5} \N(r)^{1/10+\eps}\Bigr),}
where the implied constant depends on the modulus of $\lambda$ and $\eps$.
\end{thm}

We will prove this result in Section \ref{section-Patterson} following the techniques of \cite{HB2000}. Theorem \ref{average-GS} can be compared with Theorem 1 of \cite{HB2000}  which corresponds to the case $r=1$ with a bound of $Z^{5/6+\varepsilon}$.

We first show how Theorem \ref{average-GS} implies Theorem \ref{thm-main}. Write the estimate in Theorem \ref{average-GS} as $H(Z,r,\lambda) \ll \sum_{j=1}^4 Z^{\vartheta_j} \N(r)^{\theta_j}$ with each $0 < \theta_j < 1/2 < \vartheta_j < 1$.  
We only bound $S_{11}(y)$ since the estimates for $S_{12}(y)$ and $S_{13}(y)$ follow similarly. By partial summation, we have for 
$$\lambda(\pi) = \chi_{\pi}(u^2 \sqrt{-3}\ff_\psi) \psi (\pi),$$ which is a character on $\ring$ of bounded modulus,
\bs{
&\su{\pi \equiv 1 \mod 3\\ (\pi,nd)=1 \\ \N(\pi) \le y} \frac {g(nd,\pi)} {\N(\pi)}  \lambda (\pi) \log \N(\pi)  \tilde{w} \left( \frac {X N(n)}{N(d^2 \ff_\psi) N(\pi)} \right)  \\
&\quad =\int_3^y \tilde{w} \left(\frac{X N(n)}{N(d^2 \ff_\psi) Z} \right) Z^{-1/2}  dH(Z,nd,\lambda)\\
&\quad =  \displaystyle \frac{\tilde w \left( \frac{X N(n)}{N(d^2 \ff_\psi) y} \right)} {y^{1/2}}  H(y,nd,\lambda) + \int_3^y H(Z,nd,\lambda) \biggl( \frac{\tilde{w}' \left(\frac{X N(n)}{N(d^2 \ff_\psi) Z} \right) }{Z^{5/2}} + \frac {\tilde{w} \left(\frac{X N(n)}{N(d^2 \ff_\psi) Z} \right)  }{2Z^{3/2}}\biggr) dZ
\\
&\quad \ll \sum_{j=1}^4 N(nd)^{\theta_j} \min \left( y^{\vartheta_j-1/2}, \left( \frac{X N(n)}{N(d^2 \ff_\psi)} \right)^{-2} y^{\vartheta_j+3/2} \right),
}
using Theorem \ref{average-GS}. We also used the fact that
 $\tilde{w}$ is bounded and $\tilde{w}'(x) \ll x^{-1}$ for the first bound, and the 
 fact that $\tilde{w} (x) \ll x^{-2}$ and
 $\tilde{w}' (x) \ll x^{-3}$ for the second bound.
 
Inserting this estimate in \eqref{this-is-S11} yields 
\bs{S_{11}(y) &\ll X \sum_{j =1}^4   \su{d \equiv 1 \mod 3\\ \N(d) \le D} \N(d)^{\theta_j-2}    \Biggl(\su{n \in \ring\\ \N(n) \le {y N(d^2 \ff_\psi)}/{X}} \N(n)^{\theta_j} y^{\vartheta_j-1/2} \\
&\qquad + \su{\N(n) >  y N(d^2 \ff_\psi)/X} \N(n)^{\theta_j} \(\frac{ \N( d^2\ff_\psi )}{X\N(n)}\)^2 y^{\vartheta_j+3/2} \Biggr)  \\
 &\ll X  \sum_{j=1}^4  \Biggl( y^{\vartheta_j-1/2} \su{d \equiv 1 \mod 3\\ \N(d) \le D} \N(d)^{\theta_j-2}    \( \frac{y\N(d^2)}{X} \)^{1+\theta_j}   \\
& \qquad + y^{\vartheta_j+3/2}  \su{d \equiv 1 \mod 3\\ \N(d) \le D} \N(d)^{\theta_j-2}   \(\frac{ \N( d^2 )}{X}\)^2  \su{\N(n) >   y N(d^2)/X } \N(n)^{\theta_j-2} \Biggr) \\
&\ll \sum_{j=1}^4 \frac{y^{\vartheta_j+\theta_j+1/2} D^{1+3\theta_j}}{X^{\theta_j}}.}

The estimates for $S_{12}(y)$ and $S_{13}(y)$ are similarly carried out and result in the same bound, so we conclude that
\eqn{\label{eq:S1Bound}
S_1(y) \ll \sum_{j=1}^4 \frac{y^{\vartheta_j+\theta_j+1/2} D^{1+3\theta_j}}{X^{\theta_j}}.}

\subsection{Finding the maximal support}

Combining \eqref{eq:S2bound} (under GRH) and \eqref{eq:S1Bound}, and assuming that $1 \le D \le X^{1/3}$ and $y \le X^v$, we obtain 
\bs{
S_{\sF_3'} (y) &\ll X^\eps \left( \frac{y^{4/3} D^{3/2}}{X^{1/6}} + \frac{y^{5/4} D^{7/4}}{X^{1/4}} + \frac{y^{17/12} D^{5/4} }{X^{1/12}} + \frac{y^{7/5} D^{13/10}}{X^{1/10+\eps}} \right) \\
& \quad + \frac {Xy^{1/2}}D \log^3 X + yX^{1/3}\log X.}

Using partial integration as we did at the end of the proof of Lemma \ref{lem:sizeofwholefamily}, and replacing $y=X^v$, this gives
\mult{
\eft(X) 
\ll X^\eps \bigl({X^{5v/6 - 1/6} D^{3/2}} + {X^{3v/4-1/4} D^{7/4}} + {X^{11v/12 - 1/12} D^{5/4} }\\
+ {X^{9v/10-1/10} D^{13/10}}\bigr)  + \frac {X \log^3{X}}D  + X^{v/2+{1/3}}\log X,}
and choosing $D = X^\eps$, we have that
\eqs{\eft(X) = \underline{o}(X) \quad \text{for any $v < 13/11$.}}
We remark that the bound on $v$ comes from the term $X^{11v/12 - 1/12} D^{5/4}$ above, which in turns comes from the term 
$Z^{5/6} N(r)^{1/12 + \eps}$ of  Theorem \ref{average-GS}.

This completes the proof of Theorem \ref{thm-main}, assuming Theorem \ref{average-GS}, which is proven in the next two sections.

\section{Proof of Theorem \ref{average-GS}}  \label{section-Patterson} \label{section-generalized-HB}

The proof of Theorem \ref{average-GS} is a slight generalization of the proof of Theorem 1 of  \cite{HB2000}, where the author proves the bound
\eqs{\su{c \equiv 1 \mod 3\\ \N(c) \le X}
\frac{g(c)}{\N(c)^{1/2}}  \Lambda(c) \ll X^{5/6+\varepsilon}.}
Comparing the above and the statement of Theorem \ref{average-GS}, and replacing  $g(c)=g(1,c)$ by $g(r,c) \lambda(c)$ when $(r,c)=1$, we need to keep the dependence on the shift $r$ (the character $\lambda$ has absolutely bounded conductor).

\begin{lem}[Vaughan's Identity {\cite[p. 101]{HB2000}}] \label{V-Identity}
Let $r \in \ring$, and 
\eqs{\Sigma_{j}(Z, r,u) = \sum_{a,b,c} \Lambda(a) \mu_K(b) \tilde{g}_\lambda(r,abc),\qquad (0 \le j \le 4)}
where $\tilde g_\lambda$ is defined in \eqref{H(Z,r,Lambda)} and each sum runs over $a,b,c \in \Z[\omega]$ which are square-free with $a,b,c \equiv 1 \mod 3$, $Z < N(abc) \le 2Z$,  $(r, abc)=1$, and subject to the conditions
\begin{eqnarray*}
&N(bc) \le u, \quad\quad &j=0\\
&N(b) \le u, \quad\quad &j=1\\
&N(ab) \le u, \quad\quad &j=2'\\
&N(a), N(b) \le u < N(ab) \quad\quad &j=2''\\
&N(b) \le u < N(a), N(bc) \quad\quad &j=3\\
&N(a)< N(bc) \le u, \quad\quad &j=4.\\
\end{eqnarray*}
Then,
\eqs{\Sigma_0(Z,r,u) = \Sigma_{1} (Z,r,u)- \Sigma_{2'}(Z,r,u) - \Sigma_{2''} (Z,r,u) - \Sigma_3 (Z,r,u)+\Sigma_4(Z,r,u).}
Furthermore, $$
\Sigma_0(Z,r,u) = H(2Z,r,\lambda) - H(Z,r,\lambda),
$$
and if we suppose that $1 \le u \le Z^{1/3}$, then $\Sigma_4(Z,r,u)=0$.
\end{lem}
When using Lemma \ref{V-Identity}, the sums $\Sigma_j$ are divided into the so-called Type I sums ($\Sigma_{1}$ and $\Sigma_{2'}$) and Type 2 (bilinear) sums ($\Sigma_{2''}$ and $\Sigma_{3}$), and each type is bounded with a different technique. For the Type II sums, the proof of \cite{HB2000} goes through in the exact same way, replacing $\tilde{g}(1,e)$ by $\tilde{g}_\lambda(r,e)$ with the obvious modifications, see Section \ref{sub-type2}. For the Type I sums, we have to use a general version of the work of Patterson for the distribution of the generalized Gauss sums $\tilde{g}_\lambda(r,e)$, and keep the dependence on the parameter $r$, see Section \ref{sub-type1}.

\subsection{Type I Sums} \label{sub-type1}
The so-called Type I sums of Vaughan's formula are $\Sigma_{1}(Z,r,u)$ and $\Sigma_{2'}(Z,r,u)$ and they are bounded using the work of Patterson on the distribution of the generalized Gauss sums $\tilde{g}_\lambda(r, c)$. 

\begin{lem} 
\label{type1}  
For $1 \le u \le Z^{1/3}$, 
\bs{|\Sigma_{1} (Z,r,u)| &\le 3 \log (2Z) \su{a \equiv 1 \mod 3\\(a,r)=1\\ \N(a) \le u} |\mu_K(a)| \sup_{Z\le z\le 2Z} |F_a(z,r,\lambda)| \\
|\Sigma_{2'} (Z,r,u)| &\le 2\log u \su{a\equiv 1 \mod 3\\(a,r)=1\\ \N(a) \le u} |\mu_K(a)|\sup_{Z\le z\le 2Z} |F_a(z,r,\lambda)|,}
where
\eqn{\label{eq:F}
F_a(z,r,\lambda) = \su{b \equiv 1 \mod 3\\ (r,b)=1, a \mid b\\ \N(b) \le z} \tilde{g}_\lambda(r,b).}
\end{lem}

\begin{proof}
We have 
\bs{\Sigma_1 (Z,r,u) &=  \su{N(abc) \sim Z \\ N(b) \le u\\(abc,r)=1}  \Lambda(a) \mu_K (b) \tilde{g}_\lambda(r,abc) = \su{N(bd) \sim Z \\N(b) \le u\\(bd,r)=1} \mu_K(b)  \tilde{g}_\lambda(r,bd) \sum_{a \mid d} \Lambda(a) 
 \\ &= \su{N(b) \le u\\(b,r)=1}\mu_K(b)  \su{N(d) \sim Z \\ (d,r)=1 \\b \mid d}  \tilde{g}_\lambda(r,d) \log N(d/b) ,
}
and using partial integration for the inner sum gives the first assertion. As for the second, we have
\bs{\Sigma_{2'}(Z,r,u) &=  \sum_{\substack{ N(abc) \sim Z \\ N(ab) \le u\\(abc,r)=1}}  \Lambda(a) \mu_K(b) \tilde{g}_\lambda(r,abc) \\
&= \su{ \N(d) \le u\\ (d,r)=1} \biggl(\, \su{ ab = d} \mu_K(b) \Lambda (a) \biggr) \su{c\equiv 1 \mod 3\\(c,r)=1, d\mid c\\ \N(c) \sim Z}  \tilde{g}_\lambda(r,c) \\
&\le \su{\N(d) \le u\\ (d,r)=1} |\mu_K(d) |  | F(2Z,a,r,\lambda) - F(Z,a,r,\lambda)| \log N(d),} where
on the second line we used the fact that the sum in the parenthesis is supported only on square-free $d$. This proves the second assertion. 
\end{proof}

\subsection{Type II (Bilinear) Sums} \label{sub-type2}

In this section, we bound the sums
\bs{\Sigma_{2''}(Z,r,u) &= \su{N(abc) \sim Z \\ (abc,r) = 1\\ \N(a), N(b) \le u < \N(ab)} 
\Lambda(a) \mu_K(b) \tilde{g}_\lambda(r, abc) \\
\Sigma_3 (Z,r,u) &= \su{N(abc)\sim Z \\(r,abc)=1\\ \N(a), \N(bc) >u \\ N(b) \le u} \Lambda(a) \mu_K(b) \tilde{g}_\lambda(r, abc).
}
Note that since $(abc,r)=1$, $g(r,abc) = \overline{\chi_{ab}(c)} g(r,ab) g(r,c)$ whenever $(ab,c) = 1$ and is zero otherwise, by Lemmas \ref{cubic-GS-lemma1} and \ref{cubic-GS-lemma2}. Therefore, in all cases, we have
\eqs{\tilde{g}_\lambda(r,abc) = \overline{\chi_{ab}(c)} \tilde{g}_\lambda(r,ab) \tilde{g}_\lambda(r,c),}
and we can write
\bs{\Sigma_{2''}(Z,r,u) &= \su{\N(vw) \sim Z
\\ \N(v), N(w) > u} A(v) B(w) \overline{\chi_v(w)} \\
\Sigma_{3} (Z,r,u) &= \su{\N(vw) \sim Z\\ \N(v), N(w) > u} C(v) D(w) \overline{\chi_v(w)}}
where we put
\bs{A(v) &=  \su{ab=v \\ \N(a), \N(b) \le u}  \Lambda(a) \mu_K(b) \tilde{g}_\lambda(r, ab), \quad B(w)	= \tilde{g}_\lambda(r, w)\\
C(v) &= \Lambda(v) \tilde{g}_\lambda(r, v) , \quad
D(w) = \sum_{\substack{ bc=w\\ N(b) \le u}} \mu_K(b) 	\tilde{g}_\lambda(r, bc) }
whenever $(r, vw)=1$ and zero otherwise. Notice that we have used the fact that  $u \le Z^{1/3}$  to write $N(w) > u$ in the sum for $\Sigma_{2''}$. 

Note also that 
\eqs{A(v), B(w), C(v), D(w) \le_\varepsilon X^\varepsilon,} 
for all relevant $v,w$ and that the functions are supported on square-free integers in $\Z[\omega]$ which are congruent to 1 modulo 3 by the hypothesis on $a,b,c$.
We can now use directly the proof of  \cite[Lemma 2]{HB2000} with the new functions $A,B,C,D$ as defined above, which differ 
only by fact that $\tilde{g}(v) = \tilde{g}(1,v)$ is replaced by $\tilde{g}_\lambda(r,v)$, which does not change the size.
This uses the large sieve \cite[Theorem 2]{HB2000} for cubic characters to catch the oscillation of the character $\chi_v(w)$ in the above equations for $\Sigma_{2''}$ and $\Sigma_{3}$, and we get the following.
\begin{lem} 
\label{type2}
	For any $\eps>0$, and $1 \le u \le Z^{1/3}$, we have
\eqs{\Sigma_{2''} (Z,r,u),  \; \Sigma_{3} (Z,r,u)  \ll Z^{1/2+\eps} \bigl( Z^{1/2}u^{-1/2} +Z^{1/3}\bigr),
}
where the implied constant depends only on $\eps$. In particular, choosing $u=Z^{1/3}$, this gives
\eqs{\Sigma_2 (Z,r,Z^{1/3}), \; \Sigma_3(Z,r,Z^{1/3}) \ll Z^{5/6+\eps}.}
\end{lem} 
\begin{proof}
We remark that  \cite[Lemma 2]{HB2000} states only the second part of the above lemma. For the first part, one uses the bound on the bottom of page 104  which leads
$$
\Sigma_{2''} (Z,r,u),  \; \Sigma_{3} (Z,r,u)  \ll Z^{1/2+\eps} \bigl( V^{1/2} + W^{1/2} +Z^{1/3}\bigr)
$$
together with (9) of page 103 which implies that $V, W \leq Z/u.$ \end{proof}

\subsection{Putting things together} \label{section-PTT}
Using Lemmas \ref{V-Identity}, \ref{type1} and \ref{type2}, we have for $1 \le u \le Z^{1/3}$ that 
\bs{H(2Z,r,\lambda) - H(Z,r,\lambda) &\ll \log{Z} \su{a \equiv 1 \mod 3\\(a,r)=1\\ \N(a) \le u} \sup_{Z\le z\le 2Z} |F_a(z,r,\lambda)|  \\
& \qquad + Z^{1/2+\varepsilon} \left(Z^{1/2}u^{-1/2} +Z^{1/3}\right).}

Assuming Proposition \ref{Bound4F}, which gives an upper bound for $|F_a(z,r,\lambda)|$, we deduce that
\mult{H(2Z,r,\lambda) - H(Z,r,\lambda) \ll Z^{1/2+\varepsilon} \left(Z^{1/2}u^{-1/2} +Z^{1/3}\right) \\
+ Z^{5/6+\eps} \N (r_1)^{-1/6+\eps} \N(r_3^\ast)^\eps + u^{1/2} Z^{2/3+\eps} \N(r)^{1/6+\eps} + u^{3/4} Z^{1/2+\eps} \N(r)^{\frac 14+\eps}  
.}
Finally, using Lemma \ref{lemma-kolesnik} with $u \in [1, Z^{1/3}]$ to balance the terms containing $u$ proves Theorem \ref{average-GS}, where the term $Z^{5/6+\varepsilon} N(r)^{1/12+\varepsilon}$ which gives the final estimate corresponds to $u = Z^{1/3} N(r)^{-1/6}$. 

\section{Estimate of $F_a(z,r,\lambda)$}
\label{section6}
Recall from \eqref{eq:F} that
\eqs{F_a(z,r,\lambda) = \su{b \equiv 1 \mod 3\\ a \mid b, \;(b,r)=1\\ \N(b) \le z} \tilde{g}_\lambda(r,b),}
where $\tilde{g}_\lambda (r,b)$ was defined in \eqref{H(Z,r,Lambda)}. To evaluate $F_a(z,r,\lambda)$, we will use the non-normalized generating function defined by
\begin{equation*} 
h_a (r,\lambda, s) = \su{b \equiv 1 \mod 3 \\ a\mid b, \; (b,r)=1}  g_\lambda(r,b) \N(b)^{-s},
\end{equation*}
where $g_\lambda (r,b) = \lambda(b) g(r,b) = \tilde{g}_\lambda (r,b) \N(b)^{1/2}$. The following lemma contains the analytic information on the generating function $h_a (r,\lambda, s)$.

\begin{lem}  \label{lem:sizeof_ha}
Write $r= r_1 r_2^2 r_3^3$, where $r_1, r_2$ are square-free and co-prime, and $r_i \equiv 1 \mod 3$. Let $a \equiv 1 \mod 3$ be square-free and $(a, r)=1$. Then, $h_a (r, \lambda, s+1/2)$ can be meromorphically continued to the whole complex plane; it is entire for $\Re(s) > 1/2$ except possibly for a simple pole at $s=5/6$ with residue 
\eqs{ \mathop{\res}_{s=5/6} h_a (r,\lambda,s+\textstyle{\frac12})  \ll N(a )^{-1} \N (r_1)^{-1/6} \log 2\N(ar_1) \log 2\N(r_3^\ast)}
when $r_2=1$ and $\lambda^3 \neq \lambda_0$, and zero otherwise.

Let $\eps>0$ and $\sigma_1 = 1 + \eps$. For $s =\sigma+it$ satisfying $\sigma_1 -1/2 \le \sigma \le \sigma_1$ and $|s-5/6|> 1/12$, 
\begin{equation} \label{bound-convexity}
{{h}}_a (r,\lambda, s+\textstyle{\frac12}) \ll \N(r_1r_2^2)^{\frac 12 (\sigma_1-\sigma)} N(a )^{\frac 12 (\sigma_1-\sigma)-\sigma} N(ar_1r_3^\ast )^\eps (1+t^2)^{\sigma_1-\sigma}. \end{equation}
Furthermore,
\eqn{\label{IntegralMeanValuebound4h_a}
\int_{-T}^{T} | h_a(r,\lambda, \sigma_1+it) |^2 dt \ll T^2 N(a)^{-1/2-\eps} N(r_1r_2^2r_3^\ast)^{1/2+2\eps}.}
\end{lem}

We assume Lemma \ref{lem:sizeof_ha} for now, and we prove the following proposition.
\begin{prop} Suppose $a \in \ring$ is square-free with $(a,r)=1$, and let $\lambda$ be a Dirichlet character on $\ring$. Write $r=r_1r_2^2 r_3^3$ with $r_i\equiv 1\mod 3,  \;r_1, r_2$ square-free and co-prime, and let $r_3^\ast$ be the product of the primes dividing $r_3$ but not $r_1r_2$. Then, for any $\eps>0$,  \label{Bound4F}
\bs{
F_a(z,r,\lambda)
	&\ll z^{5/6} N(a )^{-1+\eps} \N (r_1)^{-1/6+\eps} \N(r_3^\ast)^\eps + z^{2/3+\eps} N(a )^{-1/2} \N(r)^{1/6+\eps} \\
	&\qquad  + z^{1/2+\eps} N(a )^{-1/4} \N(r)^{1/4+\eps}.
}
where the first term appears only when $r_2 = 1$, and the implied constant depends on $\eps$ and the character $\lambda$.
\end{prop}
\begin{proof}[Proof of Proposition \ref{Bound4F}.]
It follows from Perron's Formula (cf. \cite[Ch.17~p.105]{Davenport}) that for $z-\frac12 \in \Z^+$ and $\sigma_1 = 1+ \eps$,
\bsc{\label{FwithPerron1}
F_a(z,r,\lambda) &- \frac{1}{2\pi i} \int_{\sigma_1 - i T}^{\sigma_1 + i T} {h}_a (r,\lambda, s+{\textstyle \frac12}) \frac{z^s}{s} ds \\
&\ll \su{b \equiv 1 \mod 3 \\ a\mid b, (b,r)=1} (z/\N(b))^{\sigma_1} \min  \bigl(1, T^{-1} |\log (z/\N(b))|^{-1}\bigr) \\
&\ll_\eps T^{-1} z^{1+\eps} N(a)^{-1}.
}
Shifting the contour to $\re s = \frac12+\eps$, we pick up the possible residue of $h_a(r,\lambda,s)$ at $s=4/3$ by Lemma \ref{lem:sizeof_ha} and therefore get
\mult{\frac{1}{2\pi i} \int_{\sigma_1 - i T}^{\sigma_1 + i T} {h}_a (r,\lambda, s+{\textstyle \frac12}) \frac{z^s}{s} ds  = 
\frac {6z^{5/6}}5 \mathop{\res}_{s=5/6} {h}_a (r,\lambda,{\textstyle \frac12}) \\
+ \frac{1}{2\pi i} \biggl(\int_{\sigma_1 - i T}^{\sigma_1-1/2 - i T} + \int_{\sigma_1-1/2 + i T}^{\sigma_1 + i T} + \int_{\sigma_1-1/2 - i T}^{\sigma_1-1/2 + i T} \biggr) {h}_a (r,\lambda, s+{\textstyle \frac12}) \frac{z^s}{s} ds.
}
Using the convexity bound \eqref{bound-convexity} of Lemma \ref{lem:sizeof_ha},
we see that 
\bsc{\label{FwithPerron2}
\biggl( &\int_{\sigma_1 - i T}^{\sigma_1-1/2 - i T} + \int_{\sigma_1-1/2 + i T}^{\sigma_1 + i T} \biggr) {h}_a (r,\lambda, s+{\textstyle \frac12}) \frac{z^s}{s} ds \\
& \ll N(ar_1r_3^\ast )^\eps\int_{\sigma_1-1/2}^{\sigma_1} T^{2(\sigma_1-\sigma)-1} \N(r_1r_2^2)^{\frac 12 (\sigma_1-\sigma)} N(a )^{\frac 12 (\sigma_1-\sigma)-\sigma}  z^\sigma d\sigma \\
& \ll N(ar_1r_3^\ast )^\eps \bigl(T^{-1} N(a )^{-1-\eps}  z^{1+\eps} + \N(r_1r_2^2)^{1/4} N(a )^{-1/4-\eps}  z^{1/2+\eps}\bigr). }
By the mean value estimate \eqref{IntegralMeanValuebound4h_a} of Lemma \ref{lem:sizeof_ha}
 and Cauchy-Schwarz inequality, we obtain 
\eqs{
\int_{-T}^{T} \left| h_a(r,\lambda, \sigma_1+it) \right| dt \ll T^{3/2} N(a)^{-1/4-\eps/2} N(r_1r_2^2r_3^\ast)^{1/4+\eps},} so that
\eqs{\int_{-T}^{T}  \Big| \frac{h_a(r,\lambda, \sigma_1+it) }{\sigma_1+it}\Big| dt \ll T^{1/2} N(a)^{-1/4-\eps/2} N(r_1r_2^2r_3^\ast)^{1/4+\eps} }
on integrating by parts. Thus, 
\eqn{\label{FwithPerron3}
\int_{\sigma_1-1/2 - i T}^{\sigma_1-1/2 + i T}  {h}_a (r,\lambda, s+{\textstyle \frac12}) \frac{z^s}{s} ds \ll T^{1/2} z^{1/2+\eps} N(a)^{-1/4-\eps/2} N(r_1r_2^2r_3^\ast)^{1/4+\eps}.}

Combining \eqref{FwithPerron1}, \eqref{FwithPerron2} and \eqref{FwithPerron3} together with the bound on the residue from Lemma \ref{lem:sizeof_ha}, we obtain
\mult{F_a(z,r,\lambda) \ll N(ar_1r_3^\ast )^\eps \bigl( T^{1/2} z^{1/2+\eps} \N(r_1r_2^2)^{1/4} N(a )^{- 1/4 -\eps}  
\\
+ T^{-1} N(a )^{-1-\eps}  z^{1+\eps} + z^{5/6} N(a )^{-1} \N (r_1)^{-1/6} \bigr) 
}
where the last term is needed only when $r_2 = 1$.

Using Lemma \ref{lemma-kolesnik} with $T \in [1,z^{1/2}]$ to bound the first two terms inside the parentheses yields the desired estimate. 
\end{proof}

\begin{rem}
Using Lemma \ref{lemma-kolesnik}, we obtain automatically a result independent of the various parameters. It would have been equivalent to choosing 
\eqs{T=z^{1/3} N(a)^{{-1/2}} N(r_1 r_2^2)^{-1/6},}
which gives the bound
\eqs{F_a(z,r,\lambda) \ll N(ar_1r_3^\ast )^\eps \Bigl( z^{2/3+\eps} {N(a)^{-1/2}} N(r_1 r_2^2)^{1/6} + z^{5/6} N(a )^{-1} \N (r_1)^{-1/6} \Bigr)}
assuming that $T \ge 1$ i.e. $N(a)^3 N(r_1 r_2^2) \le z^2$, which is true since we are taking $N(a) \le z^{1/3}$ in Section \ref{section-PTT}.\end{rem}

The rest of the section is devoted to the proof of Lemma \ref{lem:sizeof_ha}. We first state intermediate results in lemmas  \ref{lem:ha(r,lambda,s)}, \ref{lem:Patterson} and \ref{lem:from-HB}.

Our first goal is to write ${h}_a(r,\lambda,s)$ in terms of the generating function \begin{equation}  \label{def-psi}
\psi(r, \lambda,s) = \su{b \equiv 1\mod 3} \lambda(b) g(r,b) \N(b)^{-s}, \end{equation}
which appears in the work of  Patterson; namely, in \cite{Patterson-78} when $\lambda$ is trivial, and \cite{general-patterson} for the general case including the ray class character $\lambda$, following the work of Kazhdan and Patterson in \cite{K-P}.

 We also define, for any prime $\pi \equiv 1 \mod 3$,
\eqs{\psi_\pi(r, \lambda,s) = \su{b \equiv 1\mod 3\\ (b, \pi)=1} \lambda(b) g(r,b) \N(b)^{-s}.}

We now express $h_a(r,\lambda,s)$ in terms of the function $\psi(r, \lambda,s).$
Our lemma is similar to \cite[Lemma 3.6]{BaYo}, or \cite[Lemma 3.11]{DFL-1} for the function field case, where the authors of those papers are dealing with slightly different functions.
\begin{lem} \label{lem:ha(r,lambda,s)}
Suppose $a \in \ring$ is square-free with $(a,r)=1$. Write $r=r_1r_2^2 r_3^3$ with $r_i\equiv 1\mod 3, r_1, r_2$ square-free and co-prime Let $r_3^\ast$ be the product of the primes dividing $r_3$ but not $r_1r_2$. Then, 
\bs{
h_a (r,\lambda, s) &= g(r, a ) \lambda(a ) N(a )^{-s} \prod_{\pi \mid ar_1r_2} (1- \lambda(\pi)^3 \N(\pi)^{2-3s} )^{-1}  \\
	&\qquad \times \su{d\equiv 1\mod 3\\ d \mid  r_3^\ast} \frac{\mu_K(d) \lambda(d)g(ar_1r_2^2,d)}{N(d)^s} \prod_{\pi \mid d} (1- \lambda(\pi)^3 \N(\pi)^{2-3s} )^{-1}\\
	& \qquad \times \su{c \equiv 1 \mod 3\\ c\mid dar_1} \mu_K(c) \N(c)^{1-2s} \lambda(c)^2 \overline{g(dar_1r_2^2/c,c)} \psi(dar_1r_2^2/c, \lambda,s). }
\end{lem}
\begin{proof}
Recall that when $(r, b)=1$, $g(r,b)=0$ when $b$ is not square-free by Lemmata \ref{cubic-GS-lemma1} and \ref{cubic-GS-lemma2}. Then, rewriting  $b$ in the sum ${h}_a(r,\lambda,s)$ as $b=ab'$ with $b'\equiv 1 \mod 3$ and co-prime to $a$, and using Lemma \ref{cubic-GS-lemma1}, this yields
\bs{{{h}}_a (r,\lambda, s) &=
 \su{b \equiv 1 \mod 3 \\ a\mid b, \; (b, r)=1} \lambda(b) g(r,b) N(b)^{-s} \\
&= g(r, a ) \lambda(a ) N(a )^{-s}  \su{b \equiv 1 \mod 3 \\ (b, ar)=1} \lambda(b) g(ar, b) N(b)^{-s}.}
Since $(b,r)=1$ in the above sum, $g(ar, b) = \overline{\chi_b (r_3^3)} g(ar_1r_2^2,b) = g(ar_1r_2^2,b)$. Using also $(a,r)=1$, it follows that
\bs{\su{b \equiv 1 \mod 3 \\ (b, ar)=1} & \lambda(b) g(ar, b) N(b)^{-s} \\
&= \su{b \equiv 1 \mod 3 \\ (b, ar_1r_2^2)=1} \lambda(b) g(ar_1r_2^2,b) N(b)^{-s} \su{d\equiv 1\mod 3\\d \mid (b,r_3^\ast)} \mu_K(d)\\
&= \su{d\equiv 1\mod 3\\ d \mid  r_3^\ast} \frac{\mu_K(d) \lambda(d)}{N(d)^s} \su{b \equiv 1 \mod 3 \\ (bd, ar_1r_2^2)=1} \lambda(b) g(ar_1r_2^2,bd) N(b)^{-s} \\
&= \su{d\equiv 1\mod 3\\ d \mid  r_3^\ast} \frac{\mu_K(d) \lambda(d)}{N(d)^s} \su{b \equiv 1 \mod 3 \\ (b, adr_1r_2^2)=1} \lambda(b) g(ar_1r_2^2,bd) N(b)^{-s} \\
&= \su{d\equiv 1\mod 3\\ d \mid  r_3^\ast} \frac{\mu_K(d) \lambda(d)g(ar_1r_2^2,d)}{N(d)^s} \su{b \equiv 1 \mod 3 \\ (b, adr_1r_2^2)=1} \lambda(b) g(adr_1r_2^2,b) N(b)^{-s},
}
using again lemmas \ref{cubic-GS-lemma1} and \ref{cubic-GS-lemma2}.
Since $adr_1r_2$ is square-free (recall that $d \mid r_3^*$),  it follows from \cite[Lemma 3.6]{BaYo} with $r_1$ replaced by $adr_1$ that 
\bs{
\su{b \equiv 1 \mod 3 \\ (b, adr_1r_2^2)=1} & \lambda(b) g(adr_1r_2^2,b) N(b)^{-s} \\
&= \prod_{\pi \mid r_2} (1- \lambda(\pi)^3 \N(\pi)^{2-3s} )^{-1} \su{b \equiv 1 \mod 3\\ (b,adr_1)=1} \lambda(b) g(adr_1r_2^2,b) \N(b)^{-s} \\
&= \prod_{\pi \mid r_2} (1- \lambda(\pi)^3 \N(\pi)^{2-3s} )^{-1} \prod_{\pi \mid adr_1} (1- \lambda(\pi)^3 \N(\pi)^{2-3s} )^{-1} \\
&\qquad \times \su{c \equiv 1 \mod 3\\ c\mid adr_1} \mu_K((c)) \N(c)^{1-2s} \lambda(c)^2 \overline{g\left(\frac{adr_1r_2^2}{c},c\right)} \psi\left(\frac{adr_1r_2^2}{c}, \lambda,s\right).}
Combining all of the above, we arrive at the desired result.
\end{proof}

Next, we need to understand $h_a (r,\lambda,s+\frac12)$ in the strip $1/2+\eps \le \sigma \le 1+\eps$. 

\begin{lem}[Lemma p. 200 \cite{general-patterson}] \label{lem:Patterson} 
Let $r \in \Z[\omega]$. Then, ${\psi}(r,\lambda,s)$ can be meromorphically continued to the whole complex plane; it is entire for $\Re(s) > 1$, except possibly for a simple pole at $s=4/3$ with residue $\rho(r, \lambda)$ (which can occur only when $\lambda^3$ is principal). Write $r= r_1 r_2^2 r_3^3$, where $r_1, r_2$ are square-free and co-prime, and $r_i \equiv 1 \mod 3$. Then $\rho(r, \lambda)=0$ if $r_2 \neq 1$, and 
\eqs{\rho(r, \lambda) \ll N(r_1)^{-1/6},}
when $r_2 = 1$.

Let $\varepsilon > 0$, and $\sigma_1 = 3/2 + \varepsilon$. If $s =\sigma+it$,  $\sigma_1 -1/2 < \sigma < \sigma_1$, and $|s-4/3|> 1/6$, then
\eqs{{\psi}(r,\lambda,s) \ll  N(r)^{\frac 12(\sigma_1-\sigma)} (1+t^2)^{\sigma_1 - \sigma},}
where both bounds above are dependent on the conductor of the character $\lambda$.
\end{lem}

The convexity bound of the above lemma can be used to bound the integrands involved in estimating $F_a(z,r,\lambda)$, as was used in \cite{HBP1979}. Again, we are adapting the further work of \cite{HB2000} to get better bounds, by replacing pointwise bounds on the integrands by mean value bounds. Our starting point is the following lemma, which corresponds to equation (20) of \cite{HB2000} with the difference that we are considering the function $\psi(r,\lambda,s)$ defined in \eqref{def-psi}, and Heath-Brown considers only the case where $\lambda$ is trivial. The proof of the general case is identical, using the generalisations of \cite{K-P, general-patterson}.

\begin{lem}[{\cite[Equation 20]{HB2000}}] \label{lem:from-HB}
\eqs{
\int_{-T}^{T} \left\vert \psi(r,\lambda,1+\eps+it) \right\vert^2 dt \ll T^2 N(r)^{1/2}.}
\end{lem}

We remark that using the convexity bound of Lemma \ref{lem:Patterson} would lead to the weaker bound
$$
\int_{-T}^{T} \left\vert \psi(r,\lambda,1+\eps+it) \right\vert^2 dt \ll \int_{-T}^{T} \left\vert N(r)^{1/4}   t \right\vert^{2} \; dt \ll T^{3} N(r)^{1/2}.
$$

Combining the previous three lemmas we arrive at the following result for the function
$h_a(r, \lambda, s)$. 

\begin{proof}[Proof of Lemma \ref{lem:sizeof_ha}]
By Lemma \ref{lem:ha(r,lambda,s)}
\mult{
h_a (r,\lambda, s+\textstyle{\frac12}) = g(r, a ) \lambda(a ) N(a )^{-s-1/2} \prod_{\pi \mid ar_1r_2} (1- \lambda(\pi)^3 \N(\pi)^{1/2-3s} )^{-1}  \\
\times  \su{d\equiv 1\mod 3\\ d \mid  r_3^\ast} \frac{\mu_K((d)) \lambda(d)g(ar_1r_2^2,d)}{N(d)^{s+1/2}} \prod_{\pi \mid d} (1- \lambda(\pi)^3 \N(\pi)^{1/2-3s} )^{-1}\\
\times \su{c \equiv 1 \mod 3\\ c\mid dar_1} \mu_K((c)) \N(c)^{-2s} \lambda(c)^2 \overline{g(dar_1r_2^2/c,c)} \psi(dar_1r_2^2/c, \lambda,s+\frac12). }
Hence,
\mult{
\mathop{\res}_{s=5/6} h_a (r,\lambda,s+\textstyle{\frac12}) = g(r, a ) \lambda(a ) N(a )^{-4/3} \prod_{\pi \mid ar_1} (1- \lambda(\pi)^3 \N(\pi)^{-2} )^{-1}  \\
\times \su{d\equiv 1\mod 3\\ d \mid  r_3^\ast} \frac{\mu_K(d) \lambda(d)g(ar_1,d)}{N(d)^{4/3}} \prod_{\pi \mid d} (1- \lambda(\pi)^3 \N(\pi)^{-2} )^{-1}\\
\times \su{c \equiv 1 \mod 3\\ c\mid dar_1} \mu_K(c) \N(c)^{-5/3} \lambda(c)^2 \overline{g(dar_1/c,c)} \rho (dar_1/c, \lambda),}
which gives 
\bs{\mathop{\res}_{s=5/6} h_a (r,\lambda,s+\textstyle{\frac12}) &  \ll 
N(a )^{-1} \N (r_1)^{-1/6}\su{d\equiv 1\mod 3\\ d \mid  r_3^\ast} N(d)^{-1} \su{c \equiv 1 \mod 3\\ c\mid dar_1} \N(c)^{-1} \\
&\ll N(a )^{-1} \N (r_1)^{-1/6} \log 2\N(ar_1) \log 2\N(r_3^\ast).}
Using again Lemma \ref{lem:ha(r,lambda,s)}, we have for $s = \sigma+it$ as in the hypotheses,
\bs{
h_a (r,\lambda, s+\textstyle{\frac12}) &\ll \N(r_1r_2^2)^{\frac 12 (\sigma_1-\sigma)} N(a )^{\frac 12 (\sigma_1-\sigma)-\sigma} (1+t^2)^{\sigma_1-\sigma}\\
	& \qquad \times \su{d\equiv 1\mod 3\\ d \mid  r_3^\ast} N(d)^{\frac 12 (\sigma_1-\sigma)-\sigma} \su{c \equiv 1 \mod 3\\ c\mid dar_1} \N(c)^{\frac 12-2\sigma-\frac 12 (\sigma_1-\sigma)} \\
&\ll \N(r_1r_2^2)^{\frac 12 (\sigma_1-\sigma)} N(a )^{\frac 12 (\sigma_1-\sigma)-\sigma} N(ar_1r_3^\ast )^\eps (1+t^2)^{\sigma_1-\sigma},}
which proves \eqref{bound-convexity}. 
We now proceed to prove \eqref{IntegralMeanValuebound4h_a}. Again,  by Lemma \ref{lem:ha(r,lambda,s)},
\bs{
	h_a (r,\lambda, 1+\eps+it) &\ll_\eps N(a)^{-1/2-\eps} \su{d\equiv 1\mod 3\\ d \mid  r_3^\ast} \frac{|\mu_K(d)|}{N(d)^{1/2+\eps}} \\
	& \qquad \times \su{c \equiv 1 \mod 3\\ c\mid dar_1} |\mu_K(c)| \, \N(c)^{-1/2-2\eps}  \, \left|\psi(dar_1r_2^2/c, \lambda,\sigma_1+it) \right|. }
Using Cauchy-Schwarz twice, we bound $|h_a (r,\lambda, 1+\eps+it)|^2$  by
\bs{&\ll_\eps N(a)^{-1-2\eps} 
	\su{d\equiv 1\mod 3\\ d \mid  r_3^\ast} \frac{|\mu_K(d)|^2}{N(d)^{1+2\eps}} 
	 \su{d\equiv 1\mod 3\\ d \mid  r_3^\ast} \bigg| \su{c \equiv 1 \mod 3\\ c\mid dar_1} \frac{|\mu_K(c)|}{\N(c)^{1/2+2\eps}} |\psi(dar_1r_2^2/c, \lambda,\sigma_1+it)| \bigg|^2 \\
	&\ll_\eps N(a)^{-1-2\eps} 
	\su{d\equiv 1\mod 3\\ d \mid  r_3^\ast} \frac{|\mu_K(d)|^2}{N(d)^{1+2\eps}} \\
& \qquad \times \su{d\equiv 1\mod 3\\ d \mid  r_3^\ast} \biggl( \su{c \equiv 1 \mod 3\\ c\mid dar_1} 
	\frac{|\mu_K(c)|^2}{\N(c)^{1+4\eps}}  \su{c \equiv 1 \mod 3\\ c\mid dar_1} |\psi(dar_1r_2^2/c, \lambda,\sigma_1+it)|^2\biggr) \\
	&\ll_\eps N(a)^{-1-2\eps} \su{d\equiv 1\mod 3\\ d \mid  r_3^\ast} \su{c \equiv 1 \mod 3\\ c\mid dar_1} |\psi(dar_1r_2^2/c, \lambda,\sigma_1+it)|^2. }
Using Lemma \ref{lem:from-HB}, this gives
\bs{\int_{-T}^{T} | h_a(r,\lambda, \sigma_1+it) |^2 dt 
&\ll T^2 N(a)^{-1-2\eps} N(ar_1r_2^2)^{1/2} \su{d\equiv 1\mod 3\\ d \mid  r_3^\ast} N(d)^{1/2} \su{c \equiv 1 \mod 3\\ c\mid dar_1}  N(c)^{-1/2} \\
&\ll T^2 N(a)^{-1/2-\eps} N(r_1r_2^2r_3^\ast)^{1/2+2\eps},}
which proves \eqref{IntegralMeanValuebound4h_a}.
\end{proof}

\section{A positive proportion of non-vanishing} \label{section-non-vanishing}
To prove Corollary \ref{coro-non-vanishing}, we choose 
$
\phi(x) = \phi_v (x) = \displaystyle \Bigl( \frac{\sin(\pi v x)}{\pi v x} \Bigr)^2.
$
Then, 
\eqs{
\phat_v(t) =  \left\{ \begin{array}{l@{\quad}l} \displaystyle \frac{v-|t|}{v^2} & \text{if $|t| \leq v$} \\    \\ 0 & \text{otherwise} \end{array} \right.
}
is supported on $(-v, v)$. 

For $m \in \Z$, $m \ge 0$, let 
\begin{align*}
p_m(X) &= \frac{1}{\sA_{\sF_3'} (X) } \sum_{\chi \in {\sF_3'}} w\Bigl( \frac{\N(\cond(\chi))}{X} \Bigr) \delta(\chi; m) \\
\delta(\chi; m) &= \begin{cases} 1 &  \mbox{if ord}_{s=\frac12} L(s, \chi) = m \\ 0 & \mbox{otherwise.} \end{cases}
\end{align*} 
Since  $\phi_v(0)=1$ and $\phi_v(x) \geq 0$ for all $x \in \R$ and the zeros are counted with multiplicity, we have (under GRH)
\begin{align*} \frac{1}{\sA_{\sF_3'} (X) }  &\sum_{\chi \in \sF_3'} w\Bigl( \frac{\N(\cond(\chi))}{X} \Bigr)  \sum_{m=1}^\infty 
p_m(X) \\
&\le \frac{1}{\sA_{\sF_3'} (X) }  \sum_{\chi \in \sF_3'} w\Bigl( \frac{\N(\cond(\chi))}{X} \Bigr)  \sum_{m=1}^\infty m p_m(X) \\
&\le \frac{1}{\sA_{\sF_3'} (X) } \sum_{\chi \in \sF_3'} w\Bigl( \frac{\N(\cond(\chi))}{X} \Bigr) 
\sum_{\substack{\rho=\frac12 + i \gamma\\ L(\rho, \chi)=0}}  \phi_v \Bigl( \frac{\gamma \log{X}}{2 \pi} \Bigr) = \mathcal{D}(X; \phi_v, \sF_3').
\end{align*}
Since $\sum_{m\ge 0} p_m (X) = 1$,  this yields 
\eqs{p_0(X) = \frac{1}{\sA_{\sF_3'} (X) } \sum_{\substack{\chi \in \sF_3' \\ L(\frac12, \chi) \neq 0}} w\Bigl( \frac{\N(\cond(\chi))}{X} \Bigr) 
 \ge 1 -  \mathcal{D}(X; \phi_v, \sF_3') \ge 1 - \phat_v(0) +  o_X (1),}
 where the last inequality follows from Theorem  \ref{thm-main}.
 This proves a weighted version of Corollary \ref{coro-non-vanishing}. We can easily re-state this as a counting version by choosing $w$ as follows.
{Assume $X\in \Z$, and let
\eqs{w(t)= \begin{cases}
1 & 0 \le t \le 1\\ \exp\bigl(1-1/(1-X^2(x-1)^2)\bigr) & 1 < x < 1+1/X\\
0 & t \ge 1+1/X.
\end{cases}
}
Then, $w$ is smooth on $[0,\infty)$ and $\sA_{\sF_3'} (X)$ counts exactly the characters $\chi \in \sF_3'$ with $\N(\cond(\chi)) \le X$. Hence we conclude that 
\bs{\frac{ \# \{ \chi \in \sA_{\sF_3'} : \N(\cond(\chi)) \le X, L(1/2,\chi) \neq 0 \}}{\# \{ \chi \in \sA_{\sF_3'} : \N(\cond(\chi)) \le X \}} &= \frac{1}{\sA_{\sF_3'} (X) } \su{\chi \in \sF_3' \\ L(\frac12, \chi) \neq 0} w\Bigl( \frac{\N(\cond(\chi))}{X} \Bigr)    \\
&\ge 1 - \phat_v(0) +  o_X (1),}
and letting $X \to \infty$ over the integers,}
we have 
\bs{\frac{ \# \{ \chi \in \sA_{\sF_3'} : \N(\cond(\chi)) \le X, L(1/2,\chi) \neq 0 \}}{\# \{ \chi \in \sA_{\sF_3'} : \N(\cond(\chi)) \le X \}} 
 &\ge 1 - \phat_v(0) + {o}_X(1) \\
&\to 1 - \frac{11}{13} = \frac 2{13}.}


\bibliographystyle{amsplain}

\begin{thebibliography}{10}
\bibitem{BaYo} Baier, S.; Young, M. P. \emph{Mean values with cubic characters.} J. Number Theory 130 (2010), no. 4,  879 - 903.

\bibitem{BZ2008} Baier, S; Zhao, L. \emph{On the low-lying zeros of Hasse-Weil L-functions for elliptic curves.} Adv. Math. 219 (2008), no. 3, 952 - 985.

\bibitem{B1992} Brumer, A. \emph{The average rank of elliptic curves. I.}, Invent. Math. 109 (1992), no. 3, 445 - 472.

\bibitem{BuiFlo2018} Bui, H.; Florea, A. \emph{Zeros of quadratic Dirichlet L-functions in the hyperelliptic ensemble}, Transactions of the AMS (2018), no. 11, 8013 - 8045.

\bibitem{ChoPark}  Cho, P. J.; Park, J. \emph{Low-lying zeros of cubic Dirichlet L-functions and the ratios conjecture}, J. Math. Anal. Appl. 474 (2019), no. 2, 876 - 892.

\bibitem{DHP2015} David, C.; Huynh, D.K.; Parks, J. \emph{One-level density of families of elliptic curves and the Ratio Conjectures}, Research in Number Theory, 1:6, 2015, 1- 37.

\bibitem{DFL-1} David, C.; Florea, A.; Lalin, M. \emph{Mean values of cubic L-functions over function fields}, arXiv e-prints, arXiv:1901.00817 [math.NT].

\bibitem{DFL-2} David, C.; Florea, A; Lalin, M. \emph{Non-vanishing for cubic L-functions}, arXiv e-prints, arXiv:2006.15661 [math.NT].

\bibitem{Davenport} Davenport H. \emph{Multiplicative Number Theory.} Second edition. Revised by Hugh L. Montgomery. Graduate Texts in Mathematics, 74. Springer-Verlag, New York-Berlin, 1980. 


\bibitem{Diaconu} Diaconu, A. \emph{Mean square values of Hecke L-series formed with r-th order characters}, (English summary) Invent. Math. 157 (2004), no. 3, 635 - 684.

\bibitem{DPR2020} Drappeau, S.; Pratt, K.; Radziwill, M. \emph{One-level density estimates for Dirichlet L-functions with extended support},  arXiv e-prints, arXiv:2002.11968 [math.NT].

\bibitem{HM2009} Dueñez, E.; Miller, S.J. \emph{The effect of convolving families of L-functions on the underlying group symmetries}, Proc. London Math. Soc. (3) 99 (2009), no. 3, 787 - 820.

\bibitem{DM2006} Dueñez, E.; Miller, S.J. \emph{The low lying zeros of a GL(4) and a GL(6) family of L-functions}, Compos. Math. 142 (2006), no. 6, 1403 - 1425.

\bibitem{ELS2020} Ellenberg, J. S.; Li, W.; Shusterman, M. \emph{Nonvanishing of hyperelliptic zeta functions over finite fields}, Algebra Number Theory 14 (2020), no. 7, 1895 - 1909.

\bibitem{FI2003} Fouvry, E.; Iwaniec, H. \emph{Low-lying zeros of dihedral L-functions}, Duke Math. J. 116 (2003), no. 2, 189 - 217.

\bibitem{FHL} Friedberg, S.; Hoffstein, J.; Lieman, D. \emph{Double Dirichlet series and the $n$-th order twists of Hecke L-series}, Math. Ann. 327 (2003), no. 2, 315 - 338. 

\bibitem{GZ2011} Gao, P.; Zhao, L. \emph{One level density of low-lying zeros of families of L-functions}, Compos. Math. 147 (2011), no. 1, 1 - 18.

\bibitem{Gao} Gao, P. \emph{n-level density of the low-lying zeros of quadratic Dirichlet L-functions}, Int. Math. Res. Not. IMRN 2014, no. 6, 1699 - 1728. 

\bibitem{G2005} Güloğlu, A. M. \emph{Low-lying zeros of symmetric power L-functions}, Int. Math. Res. Not. (2005), no. 9, 517 - 550. 

\bibitem{GradRyzhik} Gradshteyn, I. S.; Ryzhik, I. M. \emph{Table of integrals, Series, and Products.} 4th ed., Academic Press, New York, 1965 prepared by Ju. V. Geronimus and M. Ju. Ce\u\i tlin. Translated from the Russian by Scripta Technica, Inc. Translation edited by Alan Jeffrey.

\bibitem{GraKol} Graham, S. W.; Kolesnik, G. \emph{van der Corput's method of exponential sums}, London Mathematical Society Lecture Note Series, 126. Cambridge University Press, Cambridge, 1991. vi+120 pp. ISBN: 0-521-33927-8

\bibitem{HB2000} Heath-Brown, D. R. \emph{Kummer's conjecture for cubic Gauss sums.} Israel J. Math. 120 (2000), part A, 97 - 124.

\bibitem{HBP1979} Heath-Brown, D. R.; Patterson, S. J. \emph{The distribution of Kummer sums at prime arguments}, J. Reine Angew. Math., 310 (1979), 111 - 130.

\bibitem{HB2004} Heath-Brown, D. R. \emph{The average rank of elliptic curves}, Duke Math. J. 122 (2004), no. 3, 225 - 320.

\bibitem{HR2003} Hughes, C. P.; Rudnick, Z. \emph{Linear statistics of low-lying zeros of L-functions}, Q. J. Math. 54 (2003), no. 3, 309 - 333. 

\bibitem{HM2007} Hughes C.; Miller, S. J. \emph{Low-lying zeros of L-functions with orthogonal symmetry}, Duke Math. J. 136(2007), no. 1, 115 - 172.

\bibitem{ILS}  Iwaniec, H.; Luo, W.; Sarnak, P. \emph{Low lying zeros of families of L-functions}, Inst. Hautes Etudes Sci. Publ. Math. 91(2000), 55 -131.


\bibitem{KS-book} Katz, N.; Sarnak, P. \emph{Random Matrices, Frobenius Eigenvalues and Monodromy}, AMS Colloq. Publ. 45 (1999). 

\bibitem{KS-bulletin} Katz, N.; Sarnak, P. \emph{Zeros of zeta functions and symmetry}, Bull. AMS 36 (1999), 1-26.

\bibitem{K-P} Kazhdan, D. A.; Patterson, S. J. \emph{Metaplectic Forms}, Publications Mathematiques 59 (Institut des Hautes Etudes Scientifiques, Paris, 1984), pp. 35-142.




\bibitem{Luo} Luo, W. \emph{On Hecke L-series associated with cubic characters}, Compos. Math. 140 (2004), no. 5, 1191 - 1196.

\bibitem{Li} Li, W. \emph{Vanishing of Hyperelliptic L-functions at the Central Point}, J. Number Theory 191 (2018), 85-103.

\bibitem{Meisner} Meisner,  P. \emph{One level density for cubic Galois number fields}, (English summary) Canad. Math. Bull. 62 (2019), no. 1, 149 - 167. 

\bibitem{M2008} Miller, S. J. \emph{A symplectic test of the L-functions ratios conjecture}, Int. Math. Res. Not. IMRN 2008, no. 3, 36 pp.

\bibitem{M2004} Miller, S. J. \emph{One- and two-level densities for rational families of elliptic curves: evidence for the underlying group symmetries}, Compos. Math. 140 (2004), no. 4, 952 - 992.

\bibitem{Mont} Montgomery, H. L. \emph{The Pair Correlation of Zeroes of the Zeta Function}, Proc. Sym. Pure Math., 24, AMS, 181-193, (1973). 


\bibitem{Neukirch} Neukirch, J. \emph{Algebraic Number Theory.} Translated from the 1992 German original and with a note by Norbert Schappacher. With a foreword by G. Harder. Grundlehren der Mathematischen Wissenschaften [Fundamental Principles of Mathematical Sciences], 322. Springer-Verlag, Berlin, 1999. 
	
\bibitem{OS-1999} Özlük, A. E.; Snyder, C. \emph{On the distribution of the nontrivial zeros of quadratic L-functions close to the real axis}, Acta Arith. 91 (1999), no. 3, 209 - 228.

\bibitem{general-patterson} Patterson, S. J. \emph{The distribution of general Gauss sums and similar arithmetic functions at prime arguments}, Proc. London Math. Soc. (3) 54 (2) (1987), 193 - 215.

\bibitem{Patterson-78} Patterson, S. J. \emph{On the distribution of Kummer sums}, J. Reine Angew. Math., 303 (1978), 126 - 143.


\bibitem{RR2011} Ricotta, G.; Royer, E. Statistics for low-lying zeros of symmetric power L-functions in the level aspect. Forum Math. 23 (2011), no. 5, 969-1028. 

\bibitem{Royer} Royer, E. \emph{Petits zéros de fonctions L de formes modulaires}, Acta Arith. 99 (2001), no. 2, 147-172.

\bibitem{R2001} Rubinstein, M. \emph{Low-lying zeros of L-functions and random matrix theory.} (English summary)
Duke Math. J. 109 (2001), no. 1, 147 - 181.

\bibitem{Rudnick} Rudnick, Z. \emph{Traces of high powers of the Frobenius class in the hyperelliptic ensemble},
Acta Arith. 143 (2010), no. 1, 81--99.

\bibitem{Waxman2021} Waxman, E. \emph{Lower Order Terms for the One-Level Density of a Symplectic Family of Hecke L-Functions}, J. Number Theory 221 (2021), 447 - 483.

\bibitem{Y2005} Young, M. P. \emph{Low-lying zeros of families of elliptic curves}, J. Amer. Math. Soc. 19 (2005), no. 1, 205 - 250. 



\end{thebibliography}

\end{document}